\documentclass[11pt]{amsart}
\usepackage{amsmath, amssymb, amsfonts, verbatim, graphicx, tikz-cd, mathrsfs, color, comment, fancyhdr, amsrefs}
\usepackage[top=1in, bottom=1in, left=1in, right=1in]{geometry}
\usepackage{enumerate}

\usepackage{mathtools}

\usepackage{url}

\usepackage{enumitem}

\usepackage[foot]{amsaddr}

\allowdisplaybreaks[1]

\usepackage[sc]{mathpazo}
\usepackage[T1]{fontenc}

\newcommand{\vertiii}[1]{{\left\vert\kern-0.25ex\left\vert\kern-0.25ex\left\vert #1
		\right\vert\kern-0.25ex\right\vert\kern-0.25ex\right\vert}}

\theoremstyle{plain}
\newtheorem{Thm}{Theorem}[section]
\newtheorem{Prop}[Thm]{Proposition}
\newtheorem{Lem}[Thm]{Lemma}
\newtheorem{Cor}[Thm]{Corollary}

\newtheorem*{Main results}{Main results about the gauge}

\theoremstyle{definition}
\newtheorem{Def}[Thm]{Definition}

\newtheorem{Rmk}[Thm]{Remark}

\renewcommand{\epsilon}{\varepsilon}

\usepackage{setspace}

\usepackage{hyperref}

\title{Noncommutative Ergodic Optimization and Unique Ergodicity}

\author{Aidan Young$^1$}
\address{University of North Carolina at Chapel Hill}

\email{$^1$\url{aidanjy@live.unc.edu}}
\begin{document}
	
	\maketitle
	
	\begin{abstract}
		We extend the theory of ergodic optimization and maximizing measures to the non-commutative field of C*-dynamical systems. We then employ this ergodic optimization machinery to provide an alternate characterization of unique erogdicity of C*-dynamical systems when the resident group action satisfies certain Choquet-theoretic assumptions.
	\end{abstract}
	
	One of the guiding questions of the field of ergodic optimization is the following: Given a topological dynamical system $(X, G, U)$, and a real-valued continuous function $f \in C(X)$, what values can $\int f \mathrm{d} \mu$ take when $\mu$ is an invariant Borel probability measure on $X$, and in particular, what are the extreme values it can take? In a joint work with I. Assani \cite[Section 3]{Assani-Young}, we noticed that the field of ergodic optimization was relevant to the study of certain differentiation problems dubbed \emph{spatial-temporal differentiation problems}. Hoping to extend these tools to the study of spatial-temporal differentiation problems in the setting of operator-algebraic dynamical systems, this paper develops an operator-algebraic formalization of this question of ergodic optimization, re-interpreting it as a question about the values of invariant states on a C*-dynamical system. This framework is then applied to provide a characterization of certain uniquely ergodic C*-dynamical systems with respect to ergodic optimizations.
	
	Section \ref{Ergodic optimization} develops the theory of ergodic optimization in the context of C*-dynamical systems, where the role of "maximizing measures" is instead played by invariant states on a C*-algebra. The framework we adopt is in fact somewhat more general than the classical framework of maximizing measures, even in the case where the underlying C*-algebra is commutative; however, the classical theory of ergodic optimization is still contained as a special case of our framework. We also demonstrate that some of the basic results of that classical theory of ergodic optimization extend to the C*-dynamical setting.
	
	In Section \ref{Singly generated UE}, we define a functional called the \emph{gauge} of a singly generated C*-dynamical system, a non-commutative generalization of the functional of the same name defined in \cite{Assani-Young}, and describe its connections to questions of ergodic optimization, as well as the ways in which it can be used to "detect" the unique ergodicity of C*-dynamical systems under certain Choquet-theoretic assumptions.
	
	In Section \ref{Amenable UE}, we extend the results of the previous section to the case where the phase group is a countable discrete amenable group. We also provide a characterization of uniquely ergodic C*-dynamical systems of countable discrete amenable groups, and in particular provide a new characterization of uniquely ergodic \emph{topological} dynamical systems of countable discrete amenable groups in terms of the convergence behaviors of ergodic averages.
	
	\section{Ergodic Optimization in C*-Dynamical Systems}\label{Ergodic optimization}
	
	Given a unital C*-algebra $\mathfrak{A}$, let $\operatorname{Aut}(\mathfrak{A})$ denote the family of all *-automorphisms of $\mathfrak{A}$. We endow $\operatorname{Aut}(\mathfrak{A})$ with the \emph{point-norm topology}, i.e. the topology induced by the pseudometrics
	\begin{align*}
		(\Phi, \Psi)	& \mapsto \left\| \Phi(a) - \Psi(a) \right\|	& (a \in \mathfrak{A}) .
	\end{align*}
	This topology makes $\operatorname{Aut}(\mathfrak{A})$ a topological group \cite[II.5.5.4]{Blackadar}.
	
	We define a \emph{C*-dynamical system} to be a triple $(\mathfrak{A}, G, \Theta)$ consisting of a C*-algebra $\mathfrak{A}$, a topological group $G$ (called the \emph{phase group}), and a point-continuous group action $\Theta : G \to \operatorname{Aut}(\mathfrak{A})$.
	
	Denote by $\mathcal{S}$ the family of all states on $\mathfrak{A}$ endowed with the weak*-topology, and by $\mathcal{T}$ the subfamily of all tracial states on $\mathfrak{A}$. A state $\phi$ on $\mathfrak{A}$ is called \emph{$\Theta$-invariant} (or simply \emph{invariant} if the action $\Theta$ is understood in context) if $\phi = \phi \circ \Theta_g$ for all $g \in G$. Denote by $\mathcal{S}^G \subseteq \mathcal{S}$ the family of all $\Theta$-invariant states on $\mathfrak{A}$, and by $\mathcal{T}^G \subseteq \mathcal{T}$ the family of all $\Theta$-invariant tracial states on $\mathfrak{A}$. The set $\mathcal{S}^G$ (resp. $\mathcal{T}^G$) is weak*-closed in $\mathcal{S}$ (resp. in $\mathcal{T}$). Unless otherwise stated, whenever we deal with subspaces of $\mathcal{S}$, we consider these subspaces in the weak*-topology.
	
	We will assume for the remainder of this section that $(\mathfrak{A}, G, \Theta)$ is a C*-dynamical system such that $\mathfrak{A}$ is separable, and also that $\mathcal{S}^G \neq \emptyset$. The separability of $\mathfrak{A}$ means that $\mathcal{S}$ is compact metrizable, and the assumption that $\mathcal{S}^G \neq \emptyset$ means that we can speak non-vacuously of invariant states. These assumptions are not especially burdensome. If the phase group $G$ is countable, discrete, and amenable, then we automatically have that $\mathcal{S}^G \neq \emptyset$ (see Theorem \ref{K-B}). In particular, this framework will include every system of the form $(C(Y), G, \Theta)$, where $Y$ is a compact metrizable topological space, the group $G$ is countable discrete amenable, and $\Theta_g : f \mapsto f \circ U_g$ for all $g \in G$, where $U_g$ is a homeomorphism of $Y$. More generally, if $(\mathfrak{M}, \rho, G, \Xi)$ is a W*-dynamical system with $G$ amenable, and $\mathcal{L}^2(\mathfrak{M}, \rho)$ is separable, and $(\mathfrak{A}, G, \Theta; \iota)$ is a separable C*-model of $(\mathfrak{M}, \rho, G, \Xi)$ (see Section \ref{Singly generated UE} for definitions), then $(\mathfrak{A}, G, \Theta)$ will also satisfy these hypotheses. However, both of these assumptions, i.e. that $\mathfrak{A}$ is separable and that $\mathcal{S}^G \neq \emptyset$, can fail even when $\mathfrak{A}$ is abelian.
	
	Before proceeding, we prove the following Krylov–Bogolyubov-type theorem.
	
	\begin{Thm}[Krylov–Bogolyubov Theorem for C*-dynamical systems]\label{K-B}
		Let $(\mathfrak{A}, G, \Theta)$ be a C*-dynamical system, and let $G$ be an amenable group. Let $K$ be a nonempty weak*-compact convex subset of $\mathcal{S}$ such that if $\psi \in K$, then $\psi \circ \Theta_g \in K$ for all $g \in G$. Then $K \cap \mathcal{S}^G \neq \emptyset$.
	\end{Thm}
	
	\begin{proof}
		First, fix a state $\psi \in K$, and let $(F_k)_{k = 1}^\infty$ be a left Følner sequence for $G$. Let
		$$\psi_k = \frac{1}{|F_k|} \sum_{g \in F_k} \psi \circ \Theta_g .$$
		Then each $\psi \circ \Theta_g$ is in $K$, and since $K$ is convex, it follows that $\frac{1}{|F_k|} \sum_{g \in F_k} \psi \circ \Theta_g \in K$. Let $(\psi_{k_n})_{n = 1}^\infty$ be a sub-sequence converging in the weak*-topology to some $\phi$. Then $\phi \in K$, since $K$ is compact.
		
		Now, we prove that $\phi \in \mathcal{S}^G$. Fix some $g_0 \in G, x \in \mathfrak{A}$. Then
		\begin{align*}
			\left| \phi \left( \Theta_{g_0} x \right) - \phi(x) \right|	& = \lim_{n \to \infty} \left| \psi_{k_n} \left( \Theta_{g_0} x \right) - \psi_{k_n} (x) \right| \\
			& = \lim_{n \to \infty} \frac{1}{|F_{k_n}|} \left| \left( \sum_{g \in F_{k_n}} \psi(\Theta_{g_0 g} x) \right) - \left( \sum_{g \in F_{k_n}} \psi \left( \Theta_{g} x \right) \right) \right| \\
			& = \lim_{n \to \infty} \frac{1}{|F_{k_n}|} \left| \left( \sum_{g \in g_0 F_{k_n}} \psi(\Theta_{g} x) \right) - \left( \sum_{g \in F_{k_n}} \psi \left( \Theta_{g} x \right) \right) \right| \\
			& = \lim_{n \to \infty} \frac{1}{|F_{k_n}|} \left| \left( \sum_{g \in g_0 F_{k_n} \setminus F_{k_n}} \psi(\Theta_{g} x) \right) - \left( \sum_{g \in F_{k_n} \setminus g_0 F_{k_n}} \psi \left( \Theta_{g} x \right) \right) \right| \\
			& \leq \limsup_{n \to \infty} \left(  \frac{1}{|F_{k_n}|} \left| \sum_{g \in g_0 F_{k_n}\setminus F_{k_n}} \psi(\Theta_{g} x) \right| + \frac{1}{|F_{k_n}|} \left| \sum_{g \in F_{k_n} \setminus g_0 F_{k_n}} \psi \left( \Theta_{g} x \right) \right| \right) \\
			& \leq \limsup_{n \to \infty} \frac{|g_0 F_{k_n}\Delta F_{k_n}|}{|F_{k_n}|} \| x \| \\
			& = 0 .
		\end{align*}
		Therefore $\phi \in K \cap \mathcal{S}^G$.
	\end{proof}

	Although our manner of proof of Theorem \ref{K-B} is scarcely novel, the result as we have stated it here can be used to ensure the existence of invariant states with specific properties that might interest us, as seen for example in Corollary \ref{Tracial K-B} and Proposition \ref{Nonempty annihilators}. Our standing hypothesis that $\mathfrak{A}$ be separable is not necessary for this proof of Theorem \ref{K-B}.
	
	\begin{Cor}\label{Tracial K-B}
		If $\mathcal{T} \neq \emptyset$, and $G$ is amenable, then $\mathcal{T}^G \neq \emptyset$.
	\end{Cor}
	
	\begin{proof}
		If $\phi \in \mathcal{T}$, and $x, y \in \mathfrak{A}, g \in G$, then
		\begin{align*}
			\phi(\Theta_g(x y))	& = \phi ((\Theta_g x )(\Theta_g y)) \\
			& = \phi ((\Theta_g y )(\Theta_g x)) \\
			& = \phi(\Theta_g(y x)) .
		\end{align*}
		Therefore $\phi \circ \Theta_g \in \mathcal{T}$. Thus $K = \mathcal{T}$ satisfies the hypotheses of Theorem \ref{K-B}.
	\end{proof}
	
	\begin{Def}
		We denote by $\mathfrak{R}$ the real Banach space of all self-adjoint elements of $\mathfrak{A}$, and denote by $\mathfrak{R}^\natural$ the space of all real self-adjoint bounded linear functionals on $\mathfrak{A}$.
	\end{Def}
	
	\begin{Def}\label{Simplex definition}
		Let $V$ be a locally convex topological real vector space, and let $K$ be a compact subset of $V$ which is contained in a hyperplane that does not contain the origin. We call $K$ a \emph{simplex} if the positive cone $P = \left\{ c k : c \in \mathbb{R}_{\geq 0}, k \in K \right\}$ defines a lattice ordering on $P - P = \{ p_1 - p_2 : p_1, p_2 \in P \} \subseteq V$ with respect to the partial order $a \leq b \iff b - a \in P$.
	\end{Def}
	
	\begin{Rmk}
		In Definition \ref{Simplex definition}, the assumption that $K$ lives in a hyperplane that does not contain the origin is technically superfluous, but simplifies the theory somewhat (see \cite[Section 10]{Phelps}), and is satisfied by all the simplices that interest us here. Specifically, we know that $\mathcal{S}$ (and by extension $\mathcal{S}^G, \mathcal{T}, \mathcal{T}^G$) lives in the real hyperplane $\left\{ \phi \in \mathfrak{R}^\natural : \phi(1) = 1 \right\}$ defined by the evaluation at $1$.
	\end{Rmk}
	
	We begin with the following lemma.
	
	\begin{Lem}\label{Simplices}
		\begin{enumerate}[label=(\roman*)]
			\item The spaces $\mathcal{S}, \mathcal{S}^G, \mathcal{T}, \mathcal{T}^G$ are compact and metrizable.
			\item If $\mathcal{T} \neq \emptyset$, then the space $\mathcal{T}^G$ is a simplex.
		\end{enumerate}
	\end{Lem}
	
	Before proving this lemma, we need to introduce some terminology. Let $\phi, \psi$ be two positive linear functionals on a unital C*-algebra $\mathfrak{A}$. We say that the two positive functionals are \emph{orthogonal}, notated $\phi \perp \psi$, if they satisfy either of the following two equivalent conditions:
	\begin{enumerate}[label=(\alph*)]
		\item $\| \phi + \psi \| = \| \phi \| + \| \psi \|$.
		\item For every $\epsilon > 0$ exists positive $z \in \mathfrak{A}$ of norm $\leq 1$ such that $\phi(1 - z) < \epsilon, \psi(z) < \epsilon$.
	\end{enumerate}
	It is well-know that these conditions are equivalent \cite[Lemma 3.2.3]{Pedersen}. For every $\phi \in \mathfrak{R}^\natural$, there exist unique positive linear functionals $\phi^+, \phi^-$ such that $\phi = \phi^+ - \phi^-$, and $\phi^+ \perp \phi^-$, called the \emph{Jordan decomposition} of $\phi$ \cite[II.6.3.4]{Blackadar}.
	
	Before proving Lemma \ref{Simplices}, we demonstrate the following property of the Jordan decomposition of a tracial functional.
	
	\begin{Lem}\label{Tracial Jordan}
		Let $\mathfrak{A}$ be a unital C*-algebra, and $\phi \in \mathfrak{R}^\natural$. Suppose that $\phi(x y) = \phi(y x)$ for all $x, y \in \mathfrak{A}$. Then $\phi^\pm(xy) = \phi^\pm(yx)$ for all $x, y \in \mathfrak{A}$.
	\end{Lem}
	
	\begin{proof}
		Let $\mathcal{U}(\mathfrak{A})$ denote the group of unitary elements in $\mathfrak{A}$. For a unitary element $u \in \mathcal{U}(\mathfrak{A})$, let $\operatorname{Ad}_u \in \operatorname{Aut}(\mathfrak{A})$ denote the inner automorphism
		$$\operatorname{Ad}_u x = u x u^* .$$
		Let $\psi \in \mathfrak{A}'$. We claim that $\psi$ is tracial if and only if $\psi \circ \operatorname{Ad}_u = \psi$ for all unitaries $u \in \mathcal{U}(\mathfrak{A})$.
		
		Let $u \in \mathcal{U}(\mathfrak{A})$ be unitary, and $x \in \mathfrak{A}$ an arbitrary element. Then
		\begin{align*}
			\phi(u x)	& = \psi \left(u (x u) u^* \right) \\
			& = \psi(\operatorname{Ad}_u (xu) ) .
		\end{align*}
		So $\psi(ux) = \psi(xu)$ if and only if $\psi(\operatorname{Ad}_u(xu)) = \psi(xu)$.
		
		In one direction, suppose that $\psi = \psi \circ \operatorname{Ad}_u$ for all $u \in \mathcal{U}(\mathfrak{A}).$ Fix $x, y \in \mathfrak{A}$. Then we can write $y = \sum_{j = 1}^4 c_j u_j$ for some $c_1, \ldots, c_4 \in \mathbb{C}$ and unitaries $u_1, \ldots, u_4 \in \mathcal{U} (\mathfrak{A})$ unitary. Then
		\begin{align*}
			\psi(xy)	& = \psi \left( x \sum_{j = 1}^4 c_j u_j \right) \\
			& = \sum_{j = 1}^4 c_j \psi(x u_j) \\
			& = \sum_{j = 1}^4 c_j \psi(\operatorname{Ad}_{u_j} (x u_j)) \\
			& = \sum_{j = 1}^4 c_j \psi(u_j x) \\
			& = \psi \left( \left( \sum_{j = 1}^4 c_j u_j \right) x \right) \\
			& = \psi(y x).
		\end{align*}
		Thus $\psi$ is tracial.
		
		In the other direction, suppose there exists $u \in \mathcal{U}(\mathfrak{A})$ such that $\psi \circ \operatorname{Ad}_u \neq \psi$. Let $y \in \mathfrak{A}$ such that $\psi(y) \neq \psi (\operatorname{Ad}_u y)$, and let $x = y u^*$. Then
		\begin{align*}
			\psi(x u)	& = \psi(y) \\
			& \neq \psi(\operatorname{Ad}_u y) \\
			& = \psi\left( u y u^* \right) \\
			& = \psi(u x) .
		\end{align*}
		Therefore $\psi$ is not tracial.
		
		Now, if $\phi \in \mathfrak{R}^\natural$ is tracial, then $\phi \circ \operatorname{Ad}_u = \phi$ for all $u \in \mathcal{U}(\mathfrak{A})$. Then $\phi = \phi \circ \operatorname{Ad}_u = \left(\phi^+ \circ \operatorname{Ad}_u \right) - \left(\phi^- \circ \operatorname{Ad}_u\right)$. But $\left\| \phi^\pm \circ \operatorname{Ad}_u \right\| = \left\| \phi^\pm \right\|$, so it follows that $\left\| \phi \right\| = \left\| \phi^+ \circ \operatorname{Ad}_u \right\| + \left\| \phi^- \circ \operatorname{Ad}_u \right\|$. Therefore $\phi = \left(\phi^+ \circ \operatorname{Ad}_u\right) - \left(\phi^- \circ \operatorname{Ad}_u\right)$ is an orthogonal decomposition of $\phi$, and so it is \emph{the} Jordan decomposition. This means that $\phi^\pm = \phi^\pm \circ \operatorname{Ad}_u$. Since this is true for all $u \in \mathcal{U}(\mathfrak{A})$, it follows that $\phi^\pm$ are tracial.
	\end{proof}
	
	\begin{proof}[Proof of Lemma \ref{Simplices}]
		\begin{enumerate}[label=(\roman*)]
			\item This all follows because $\mathcal{S}$ is a weak*-closed real subspace of the unit ball in the continuous dual of the separable Banach space $\mathfrak{R}$, and the spaces $\mathcal{S}^G, \mathcal{T}, \mathcal{T}^G$ are all closed subspaces of $\mathcal{S}$.
			
			\item It is a standard fact that if $\mathcal{T} \neq \emptyset$, then $\mathcal{T}$ is a simplex \cite[II.6.8.11]{Blackadar}. Let
			$$C^G = \left\{ c \phi : c \in \mathbb{R}_{\geq 0} , \phi \in \mathcal{T}^G \right\}$$ be the positive cone of $\mathcal{T}^G$, and let $\mathfrak{R}^\natural$ denote the (real) space of all bounded self-adjoint tracial linear functionals on $\mathfrak{A}$. Let $E^G$ denote the (real) space of all bounded self-adjoint $\Theta$-invariant linear functionals on $\mathfrak{A}$. We already know that $\mathcal{T}$ lives in a hyperplane of $\mathfrak{R}^\natural$ defined by the evaluation functional $\phi \mapsto \phi(1)$. It will therefore suffice to show that $E^G = C^G - C^G$, and that $E^G$ is a sub-lattice of $\mathfrak{R}^\natural$.
			
			Let $\phi^+, \phi^- \geq 0$ be positive functionals on $\mathfrak{A}$ such that $\phi = \phi^+ - \phi^-$ is tracial, and $\phi^+ \perp \phi^-$. By Lemma \ref{Tracial Jordan}, we know that $\phi^+, \phi^-$ are tracial. We claim that if $\phi \in E^G$, then $\phi^+, \phi^- \in C^G$. To prove this, let $g \in G$, and consider that $\phi^+ \circ \Theta_g, \phi^- \circ \Theta_g$ are both positive linear functionals such that $\phi = \left( \phi^+ \circ \Theta_g \right) - \left( \phi^- \circ \Theta_g \right)$.
			
			We claim that $\left( \phi^+ \circ \Theta_g \right) \perp \left( \phi^- \circ \Theta_g \right)$. Fix $\epsilon > 0$. We know that there exists $z \in \mathfrak{A}$ such that $\| z \| \leq 1, 0 \leq z$, and such that $\phi^+ \left( 1 - z \right) < \epsilon, \phi^- (z) < \epsilon$. Then $\Theta_{g^{-1}} (z)$ is a positive element of norm $\leq 1$ such that
			\begin{align*}
				\phi^+ \left( \Theta_g \left( \Theta_{g^{-1}} (1 - z) \right) \right)	& = \phi^+ (1 - z)	& < \epsilon, \\
				\phi^- \left( \Theta_g \left( \Theta_{g^{-1}} (z) \right) \right)	& = \phi^{-} (z)	& < \epsilon .
			\end{align*}
			Therefore $\left( \phi^+ \circ \Theta_g \right) - \left( \phi^- \circ \Theta_g \right)$ is a Jordan decomposition of $\phi$, and since the Jordan decomposition is unique, it follows that $\phi^+ = \phi^+ \circ \Theta_g, \phi^- = \phi^- \circ \Theta_g$, i.e. that $\phi^+, \phi^- \in C^G$. This means that $E^G = C^G - C^G$.
			
			We now want to show that $E^G = C^G - C^G$ is a sublattice of $E$, i.e. that it is closed under the lattice operations. Let $\phi, \psi \in E^G$. For this calculation, we draw on the identities listed in \cite[Theorem 1.3]{PositiveOperators}. Then
			\begin{align*}
				\phi \lor \psi	& = \left( \left( (\phi - \psi) + \psi \right) \lor (0 + \psi) \right) \\
				& = \left( (\phi - \psi) \lor 0 \right) + \psi \\
				& = (\phi - \psi)^+ + \psi , \\
				\phi \land \psi	& = \left( (\phi - \psi) + \psi \right) \land (0 + \psi) \\
				& = \left( (\phi - \psi) \land 0 \right) + \psi \\
				& = - \left( (- (\phi - \psi)) \lor 0 \right) + \psi \\
				& = - (\psi - \phi)^+ + \psi .
			\end{align*}
			Therefore, if $E^G$ is a real linear space and is closed under the operations $\phi \mapsto \phi^+, \phi \mapsto \phi^-$, then it is also closed under the lattice operations. Thus $E^G$ is a sublattice of $\mathfrak{R}^\natural$.
			
			Hence, the subset $\mathcal{T}^G$ is a compact metrizable simplex.
		\end{enumerate}
	\end{proof}
	
	In order to keep our treatment relatively self-contained, we define here several elementary concepts from Choquet theory that will be relevant in this section.
	
	\begin{Def}
		Let $S_1, S_2$ be convex spaces. We call a map $T : S_1 \to S_2$ an \emph{affine map} if for every $v, w \in S_1 ; \; t \in [0, 1]$, we have
		$$T(t v + (1 - t) w) = t T(v) + (1 - t) T(w) .$$
		In the case where $S_2 \subseteq \mathbb{R}$, we call $T$ an \emph{affine functional}.
	\end{Def}
	
	\begin{Def}
		Throughout, let $K$ be a convex subset of a locally convex real topological vector space $V$.
		\begin{enumerate}[label=(\alph*)]
			\item A point $k \in K$ is called an \emph{extreme point} of $K$ if for every pair of points $k_1, k_2 \in K$ and parameter $t \in [0, 1]$ such that $k = t k_1 + (1 - t) k_2$, either $k_1 = k_2$ or $t \in \{0, 1\}$. In other words, we call $k$ extreme if there is no nontrivial way of expressing $k$ as a convex combination of elements of $K$.
			\item The set of all extreme points of $K$ is denoted $\partial_e K$.
			\item A subset $F$ of $K$ is called a \emph{face} if for every pair $k_1, k_2 \in K, t \in (0, 1)$ such that $t k_1 + (1 - t) k_2 \in F$, we have that $k_1, k_2 \in F$.
			\item A face $F$ of $K$ is called an \emph{exposed face} of $K$ if there exists a continuous affine functional $\ell : K \to \mathbb{R}$ such that $\ell(x) = 0$ for all $x \in F$, and $\ell(y) < 0$ for all $y \in K \setminus F$.
			\item A point $k \in K$ is called an \emph{exposed point} of $K$ if $\{k\}$ is an exposed face of $K$.
			\item Given a subset $\mathcal{E}$ of $K$, the closed convex hull of $\mathcal{E}$ is written as $\overline{\operatorname{co}}(\mathcal{E})$.
		\end{enumerate}
	\end{Def}
	
	We now introduce the basic concepts in our treatment of ergodic optimization.
	
	\begin{Def}
		Let $x \in \mathfrak{R}$ be a self-adjoint element, and let $K \subseteq \mathcal{S}^G$ be a compact convex subset of $\mathcal{S}^G$. Define a value $m \left( x \vert K \right)$ by
		$$m \left( x \vert K \right) : = \sup_{\psi \in K} \psi(x) .$$
		We say a state $\phi \in K$ is \emph{$(x \vert K)$-maximizing} if $\phi(x) = m (x \vert K)$. Let $K_\mathrm{max}(x) \subseteq K$ denote the set of all $(x \vert K)$-maximizing states. A state $\phi \in K$ is called \emph{uniquely $(x \vert K)$-maximizing} if $K_\mathrm{max}(x) = \{\phi\}$.
	\end{Def}
	
	\begin{Rmk}
		We note here the following, somewhat obvious inequality: If $K_1 \subseteq K_2$ are compact convex subsets of $\mathcal{S}^G$, then $m \left( x \vert K_1 \right) \leq m \left( x \vert K_2 \right)$, and in particular, we will always have $m \left( x \vert K_1 \right) \leq m \left( x \vert \mathcal{S}^G \right)$.
	\end{Rmk}
	
	We will single out one type of compact convex subset of $\mathcal{S}^G$ which will prove important later. Given a subset $A \subseteq \mathfrak{A}$, set $$\operatorname{Ann}(A) : = \left\{ \phi \in \mathcal{S}^G : A \subseteq \ker \phi \right\} .$$
	
	When $\mathfrak{I} \subseteq \mathfrak{A}$ is a $\Theta$-invariant closed ideal of $\mathfrak{A}$, we have a bijective correspondence between the states in $\operatorname{Ann}(\mathfrak{I})$ and the states on $\mathfrak{A} / \mathfrak{I}$ invariant under the action induced by $\Theta$. We will be referring to this set again in Sections \ref{Singly generated UE} and \ref{Amenable UE}, when values of the form $m \left( a \vert \operatorname{Ann}(A) \right)$ come up in reference to certain ergodic averages. We observe that $\operatorname{Ann}(\{0\}) = \mathcal{S}^G$, and that $A \subseteq B \subseteq \mathfrak{A} \Rightarrow \operatorname{Ann}(A) \supseteq \operatorname{Ann}(B)$. There is also no a priori guarantee that $\operatorname{Ann}(A) \neq \emptyset$, since for example $\operatorname{Ann}(\{1\}) = \emptyset$. However, Proposition \ref{Nonempty annihilators} gives sufficient conditions for $\operatorname{Ann}(A)$ to be nonempty.
	
	\begin{Prop}\label{Nonempty annihilators}
	Let $A \subseteq \mathfrak{A}$ be such that $\Theta_g A \subseteq A$ for all $g \in G$. Suppose there exists a state on $\mathfrak{A}$ which vanishes on $A$. Then $\operatorname{Ann}(A) \neq \emptyset$. In particular, if $\mathfrak{I} \subsetneq \mathfrak{A}$ is a proper closed two-sided ideal of $\mathfrak{A}$ for which $\Theta_g \mathfrak{I} = \mathfrak{I}$ for all $g \in G$, then $\operatorname{Ann}(\mathfrak{I}) \neq \emptyset$.
	\end{Prop}

	\begin{proof}
	Let $K \subseteq \mathcal{S}$ denote the family of all (not necessarily invariant) states on $\mathfrak{A}$ which vanish on $A$. Then if $\phi \in K$ and $a \in A$, then $\Theta_g a \in A$, so $\phi \circ \Theta_g$ vanishes on $A$. Therefore $\Theta_g K \subseteq K$ for all $g \in G$. It follows from Theorem \ref{K-B} that $K \cap \mathcal{S}^G = \operatorname{Ann}(A) \neq \emptyset$.
	
	Suppose $\mathfrak{I} \subsetneq \mathfrak{A}$ is a proper closed two-sided ideal of $\mathfrak{A}$ for which $\Theta_g \mathfrak{I} = \mathfrak{I}$ for all $g \in G$, and let $\pi : \mathfrak{A} \twoheadrightarrow \mathfrak{A} / \mathfrak{I}$ be the canonical quotient map. Let $\tilde{\Theta} : G \to \operatorname{Aut}(\mathfrak{A} / \mathfrak{I})$ be the induced action of $G$ on $\mathfrak{A} / \mathfrak{I}$ by $\tilde{\Theta}_g(a + \mathfrak{I}) = \Theta_g a + \mathfrak{I}$. Let $\psi$ be a $\tilde{\Theta}$-invariant state on $\mathfrak{A} / \mathfrak{I}$. Then $\psi \circ \pi$ is a $\Theta$-invariant state on $\mathfrak{A}$ which vanishes on $\mathfrak{I}$, i.e. $\psi \circ \pi \in \operatorname{Ann}(\mathfrak{I})$.
	\end{proof}
	
	\begin{Prop}\label{Nonempty maximizing states}
		Let $K \subseteq \mathcal{S}^G$ be a nonempty compact convex subset of $\mathcal{S}^G$, and let $x \in \mathfrak{R}$. Then $K_\mathrm{max}(x)$ is a nonempty, compact, exposed face of $K$.
	\end{Prop}
	
	\begin{proof}
		To see that $K_\mathrm{max}(x)$ is nonempty, for each $n \in \mathbb{N}$, let $\phi_n \in K$ such that $\phi_n(x) \geq m(x \vert K) - \frac{1}{n}$. Then since $K$ is compact, the sequence $(\phi_n)_{n = 1}^\infty$ has a convergent subsequence. Let $\phi$ be the limit of a convergent subsequence of $(\phi_n)_{n = 1}^\infty$. Then $\phi$ is $(x \vert K)$-maximizing.
		
		To see that $K_\mathrm{max}(x)$ is compact, consider that
		$$K_\mathrm{max}(x) = \left\{ \phi \in K : \phi(x) = m (x \vert K) \right\} ,$$
		which is a closed subset of $K$. As for being an exposed face, consider the continuous affine functional $\ell : K \to \mathbb{R}$ given by
		$$\ell(\phi) = \phi(x) - m(x \vert K) .$$
		Then the functional $\ell$ exposes $K_\mathrm{max}(x \vert K)$, since it is nonpositive on all of $K$ and vanishes exactly on $K_\mathrm{max}(x)$.
	\end{proof}

	The following result describes the ways in which some ergodic optimizations interact with equivariant *-homomorphisms of C*-dynamical systems.
	
	\begin{Thm}\label{Ergodic optimization through *-homomorphisms}
		Let $\left( \mathfrak{A} , G , \Theta \right) , \left( \tilde{\mathfrak{A}} , G , \tilde{\Theta} \right)$ be two C*-dynamical systems, and let $\pi : \mathfrak{A} \to \tilde{\mathfrak{A}}$ be a surjective *-homomorphism such that
		\begin{align*}
			\tilde{\Theta}_g \circ \pi	& = \pi \circ \Theta_g	& (\forall g \in G) .
		\end{align*}
		Let $\tilde{\mathcal{S}}^G$ denote the space of $\tilde{\Theta}$-invariant states on $\tilde{\Theta}$. Then $m \left( \pi(a) \vert \tilde{\mathcal{S}}^G \right) = m \left( a \vert \operatorname{Ann}(\ker \pi) \right)$.
	\end{Thm}
	
	\begin{proof}
		Let $\tilde{\mathcal{S}}^G$ denote the space of $\tilde{\Theta}$-invariant states on $\tilde{\mathfrak{A}}$. We claim that there is a natural bijective correspondence between $\tilde{\mathcal{S}}^G$ and $\operatorname{Ann}(\ker f)$. If $\phi$ is a $\tilde{\Theta}$-invariant state on $\tilde{\mathfrak{A}}$, then we can pull it back to a $\Theta$-invariant state $\phi_0$ on $\mathfrak{A}$ by
		$$\phi_0 = \phi \circ \pi .$$
		This $\phi_0$ obviously vanishes on $\ker \pi$, and is $\Theta$-invariant by virtue of the equivariance property of $\pi$. Conversely, if we start with a $\Theta$-invariant state $\psi$ on $\mathfrak{A}$ that vanishes on $\ker \pi$, then we can push it to a $\tilde{\Theta}$-invariant state $\tilde{\psi}$ on $\tilde{ \mathfrak{A} }$ by
		$$\tilde{\psi} \circ \pi = \psi .$$
		
		We claim now that
		$$m \left( a \vert \operatorname{Ann}(\ker \pi) \right) = m \left( \pi(a) \vert \tilde{\mathcal{S}}^G \right).$$
		Let $\phi$ be a $\left( \pi(a) \vert \tilde{\mathcal{S}}^G \right)$-maximizing state on $\tilde{ \mathfrak{A} }$. Then $\phi \circ \pi \in \operatorname{Ann}(\ker \pi)$, so
		$$m \left( \pi(a) \vert \tilde{\mathcal{S}}^G \right) = \phi(\pi(a)) \leq m \left( a \vert \operatorname{Ann}(\ker \pi) \right) .$$
		On the other hand, if $\psi \in \operatorname{Ann}(\ker \pi)$ is $\left( a \vert \operatorname{Ann}(\ker \pi) \right)$-maximizing, then let $\tilde{\psi}$ be such that $\tilde{\psi} \circ \pi = \psi$. Then $\tilde{\psi} \in \tilde{\mathcal{S}}^G$, so
		$$m \left( a \vert \operatorname{Ann}(\ker \pi) \right) = \psi(a) = \tilde{\psi}(\pi(a)) \leq m \left( a \vert \tilde{\mathcal{S}}^G \right) .$$
	\end{proof}

	The assumption in Theorem \ref{Ergodic optimization through *-homomorphisms} that $\pi$ is surjective is actually superfluous, as shown in Corollary \ref{Ergodic optimization through non-surjective *-homomorphisms}. We will later provide a proof of this stronger claim that uses the gauge functional, introduced in the context of actions of $\mathbb{Z}$ in Section \ref{Singly generated UE} and in the context of actions of amenable groups in Section \ref{Amenable UE}.
	
	Moreover, the proof of Theorem \ref{Ergodic optimization through *-homomorphisms} can be extended to establish a correspondence between ergodic optimization over certain compact convex subsets of $\tilde{\mathcal{S}}^G$ and certain compact convex subsets of $\operatorname{Ann}(\ker \pi)$. For example under the same hypotheses, if $\mathcal{T} \neq \emptyset$, then the proof could be modified in a simple manner to establish that $m \left( \pi(a) \vert \tilde{\mathcal{T}}^G \right) = m \left( a \vert \operatorname{Ann}(\ker \pi) \cap \mathcal{T}^G \right)$, where $\tilde{\mathcal{T}}^G$ denotes the $\tilde{\Theta}$-invariant tracial states on $\tilde{\mathfrak{A}}$. In lieu of stating Theorem \ref{Ergodic optimization through *-homomorphisms} in greater generality, we content ourselves to state this special case (which we will use in future sections) and remark that the argument can be generalized further.
	
	The following characterization of exposed faces in compact metrizable simplices will prove useful.
	
	\begin{Lem}\label{Closed faces are exposed}
		Let $K$ be a compact metrizable simplex. Then every closed face of $K$ is exposed.
	\end{Lem}
	
	\begin{proof}
		See \cite[Theorem 7.4]{Davies}.
	\end{proof}
	
	The theorem we are building to in this section is as follows.
	
	\begin{Thm}\label{Simplices and exposed faces}
	Let $K \subseteq \mathcal{S}^G$ be a compact simplex. Then the closed faces of $K$ are exactly the sets of the form $K_\mathrm{max}(x)$ for some $x \in \mathfrak{R}$.
	\end{Thm}
	
	Before we can prove our main theorem of this section, we will need to prove the following result, which gives us a means by which to build an important linear functional.
	
	\begin{Thm}\label{Extension Theorem}
		Let $K \subseteq \mathcal{S}^G$ be a compact simplex, and let $\ell : K \to \mathbb{R}$ be a continuous affine functional. Then there exists a continuous linear functional $\tilde{\ell} : \overline{\operatorname{span}}_\mathbb{R}(K) \to \mathbb{R}$ such that $\tilde{\ell} \vert_{K} = \ell$.
	\end{Thm}
	
	To prove this theorem, we break it up into several parts, attaining the extension $\tilde{\ell}$ as the final step of a few subsequent extensions of $\ell$.
	
	\begin{Lem}\label{First extension}
		Let $K \subseteq \mathcal{S}^G$ be a compact metrizable simplex, and let $\ell : K \to \mathbb{R}$ be a continuous affine functional. Let $P = \left\{ c \phi : c \in \mathbb{R}_{\geq 0}, \phi \in K \right\}$. Then there exists a continuous functional $\ell_1 : P \to \mathbb{R}$ satisfying the following conditions for all $f_1, f_2 \in P ; c \in \mathbb{R}_{\geq 0}$:
		\begin{enumerate}[label=(\alph*)]
			\item $\ell_1(c f_1) = c \ell_1(f_1)$,
			\item $\ell_1(f_1 + f_2) = \ell_1(f_1) + \ell_2(f_2)$,
			\item $\ell_1 \vert_K = \ell$.
		\end{enumerate}
	\end{Lem}
	
	\begin{proof}
		Note that every nonzero element of $P$ can be expressed uniquely as $c \phi$ for some $c \in \mathbb{R}_{\geq 0} \setminus \{0\} , \phi \in K$. As such we define
		$$\ell_1(c \phi) = \begin{cases}
			c \ell(\phi)	& c > 0 \\
			0	& c = 0
		\end{cases}$$
		It is immediately clear that this $\ell_1$ satisfies conditions (a) and (c), leaving only (b) to check.
		
		Now, suppose that $f_1 = c_1 \phi_1, f_2 = c_2 \phi_2$ for some $\phi_1, \phi_2 \in K ; c_1, c_2 \in \mathbb{R}_{\geq 0}$. Consider first the case where at least one of $c_1, c_2$ are nonzero.
		Then
		\begin{align*}
			f_1 + f_2	& = c_1 \phi_1 + c_2 \phi_2	\\
			& = (c_1 + c_2) \left( \frac{c_1}{c_1 + c_2} \phi_1 + \frac{c_2}{c_1 + c_2} \phi_2 \right) \\
			\Rightarrow \ell_1(f_1 + f_2)	& = \ell_1 \left( c_1 \phi_1 + c_2 \phi_2 \right)	\\
			& = (c_1 + c_2) \ell \left( \frac{c_1}{c_1 + c_2} \phi_1 + \frac{c_2}{c_1 + c_2} \phi_2 \right) \\
			& = (c_1 + c_2) \left( \frac{c_1}{c_1 + c_2} \ell(\phi_1) + \frac{c_2}{c_1 + c_2} \ell(\phi_2) \right)	& \textrm{(because $\ell$ is affine)} \\
			& = c_1 \ell(\phi_1) + c_2 \ell(\phi_2) \\
			& = \ell_1(c_1 \phi_1) + \ell_1(c_2 \phi_2) \\
			& = \ell_1(f_1) + \ell_1(f_2) .
		\end{align*}
		In the event that $c_1 = c_2 = 0$, then the additivity property attains trivially.
		
		It remains now to show that $\ell_1$ is continuous. We will check continuity at nonzero points in $P$, and then at $0 \in P$. First, consider the case where $c \phi \in P \setminus \{0\}$, and $c \in \mathbb{R}_{\geq 0}, \phi \in K$. Suppose that $(c_n \phi_n)_n$ is a sequence in $P$ converging in the weak*-topology to $c \phi$. We claim that $c_n \to c$ in $\mathbb{R}$, and $\phi_n \to \phi$ in the weak*-topology.
		
		We first observe that $(c_n \phi_n)(1) = c_n$, so $(c_n)_{n}$ converges in $\mathbb{R}_{\geq 0}$ to $c$, meaning in particular that for sufficiently large $n$, we have that $c_n \in \left[ \frac{c}{2} , \frac{3c}{2} \right] $. Now, if $\lambda : \mathfrak{R} \to \mathbb{R}$ is a norm-continuous linear functional, then
		\begin{align*}
			\lambda(\phi_n)	& = \frac{1}{c_n} \lambda(c_n \phi_n) \\
			& \to \frac{1}{c} \lambda(c \phi) \\
			& = \lambda(\phi) .
		\end{align*}
		Therefore $c_n \to c , \phi_n \to \phi$. Thus we can compute
		\begin{align*}
			\left| \ell_1(c \phi) - \ell_1(c_n \phi_n) \right|	& \leq \left| \ell_1(c \phi) - \ell_1(c_n \phi) \right| + \left| \ell_1(c_n \phi) - \ell_1(c_n \phi_n) \right| \\
			& = |c - c_n| \cdot \left| \ell(\phi) \right| + |c_n| \cdot \left| \ell(\phi) - \ell(\phi_n) \right| \\
			& \leq |c - c_n| \left( \sup_{\phi \in K} |\ell(\phi)| \right) + \frac{3c}{2} \left| \ell(\phi) - \ell(\phi_n) \right| \\
			& \to 0 ,
		\end{align*}
		where the supremum invoked must exist because $K$ is weak*-compact, and $|\ell(\phi) - \ell(\phi_n)| \to 0$ because $\ell$ is weak*-continuous.
		
		Now, suppose that $(c_n \phi_n)_n$ converges to $0$. Then again we have that $c_n \to 0$ by the same argument used above (i.e. $c_n = (c_n \phi_n)(1)$). Therefore
		\begin{align*}
			\left| \ell_1 (c_n \phi_n) \right|	& = |c_n| \cdot \left| \ell (\phi_n) \right| \\
			& \leq |c_n| \left( \sup_{\phi \in K} |\ell(\phi)| \right) \\
			& \to 0 .
		\end{align*}
		
		We can thus conclude that $\ell_1$ is weak*-continuous.
	\end{proof}
	
	\begin{Lem}\label{Second extension}
		Let $\ell_1, P$ be as in Lemma \ref{First extension}, and let $V = P - P$. Then there exists a continuous linear functional $\tilde{\ell} : V \to \mathbb{R}$ such that $\tilde{\ell} \vert_{P} = \ell_1$.
	\end{Lem}
	
	\begin{proof}
		Define $\tilde{\ell} : V \to \mathbb{R}$ by
		$$\tilde{\ell}(v) = \ell_1 \left( v^+ \right) - \ell_1 \left( v^- \right),$$
		where $v^+, v^-$ are meant in the sense of the lattice structure $V$ possesses by virtue of $K$ being a simplex.
		
		Our first claim is that if $f, g \in P$ such that $v = f - g$, then $\tilde{\ell}(v) = \ell_1(f) - \ell_1(g)$. To see this, we observe that $f + v^- = g + v^+ \in P$. Therefore
		\begin{align*}
			\ell_1 \left( f + v^- \right)	& = \ell_1 \left( g + v^+ \right) \\
			= \ell_1(f) + \ell_1 \left( v^- \right)	& = \ell_1(g) + \ell_1 \left( v^+ \right) \\
			\Rightarrow \ell_1(f) - \ell_1(g)	& = \ell_1 \left( v^+ \right) - \ell_1 \left( v^- \right) \\
			& = \tilde{\ell}(v) .
		\end{align*}
		This makes linearity fairly straightforward to check. First, to confirm additivity, let $v, w \in V$. Then $v + w = \left( v^+ + w^+ \right) - \left( v^- + w^- \right)$, where $v^+ + w^+, v^- + w^- \in P$. Thus
		\begin{align*}
			\tilde{\ell}(v + w)	& = \ell_1 \left( v^+ + w^+ \right) - \ell_1 \left( v^- + w^- \right) \\
			& = \ell_1 \left( v^+ \right) + \ell_1 \left( w^+ \right) - \ell_1 \left( v^- \right) - \ell_1 \left( w^- \right) \\
			& = \ell_1 \left( v^+ \right) - \ell_1 \left( v^- \right) + \ell_1\left( w^+ \right) - \ell_1 \left( w^- \right) \\
			& = \tilde{\ell}(v) + \tilde{\ell}(w) .
		\end{align*}
		To check homogeneity, let $c \in \mathbb{R}$. If $c \geq 0$, then $cv^+, c v^- \in P$, and $c v^+ - c v^- = c v$; on the other hand, if $c \leq 0$, then $- c v^- , - c v^+ \in P$, and $c v = - c v^- + c v^+$. In both cases, homogeneity is straightforward to show. This proves that $\tilde{\ell}$ is linear.
		
		It is also quick to show that $\tilde{\ell} \vert_{P} = \ell_1$, since if $v \in P$, then $v = v^+$, so $\tilde{\ell}(v) = \ell_1 \left( v^+ \right) - 0 = \ell_1(v)$.
		
		It remains now to show that $\tilde{\ell}$ is continuous. By \cite[Theorem 1.18]{RudinFunctional}, it will suffice to show that $\ker \tilde{\ell}$ is weak*-closed. To prove the kernel is closed, let $(v_n)_{n = 1}^\infty$ be a sequence in $\ker \tilde{\ell}$ converging in the weak*-topology to $v \in V$. By the Uniform Boundedness Principle, it follows that $\sup_{n} \| v_n \| < \infty$. By rescaling, we can assume without loss of generality that $\|v_n\| \leq 1$ for all $n \in \mathbb{N}$, and since the unit ball $B \subseteq V$ is weak*-closed by Banach-Alaoglu, we can infer that $\| v \| \leq 1$.
		
		Since the unit ball $B$ is weak*-compact, it follows that the sequences $\left( v_n^+ \right)_{n = 1}^\infty , \left( v_n^- \right)_{n = 1}^\infty$ have convergent subsequences. Let $(n_j)_{j = 1}^\infty$ be a subsequence along which $v_{n_j}^+ \to m_1 \in P, v_{n_j}^- \to m_2 \in P$. Then if $x \in \mathfrak{R}$, then
		\begin{align*}
			v(x)	& = \lim_{n \to \infty} v_n(x) \\
			& = \lim_{n \to \infty} \left( v_n^+(x) - v_n^-(x) \right) \\
			& = \lim_{j \to \infty} \left( v_{n_j}^+(x) - v_{n_j}^-(x) \right) \\
			& = \left( \lim_{j \to \infty} v_{n_j}^+(x) \right) - \left( \lim_{j \to \infty} v_{n_j}^-(x) \right) \\
			& = m_1(x) - m_2(x) .
		\end{align*}
		Therefore $v = m_1 - m_2$, so
		\begin{align*}
			\tilde{\ell}(v)	& = \tilde{\ell} (m_1) - \tilde{\ell} (m_2) \\
			& = \left( \lim_{j \to \infty} \tilde{\ell} \left( v_{n_j}^+ \right) \right) - \left( \lim_{j \to \infty} \tilde{\ell} \left( v_{n_j}^- \right) \right) \\
			& = \lim_{j \to \infty} \left( \tilde{\ell} \left( v_{n_j}^+ \right) - \tilde{\ell} \left( v_{n_j}^- \right) \right) \\
			& = \lim_{j \to \infty} \tilde{\ell} \left( v_{n_j} \right) \\
			& = \lim_{j \to \infty} 0	\\
			& = 0 .
		\end{align*}
		Therefore, we can conclude that $\tilde{\ell}$ is weak*-continuous.
	\end{proof}
	
	\begin{proof}[Proof of Theorem \ref{Extension Theorem}]
		This follows from Lemmas \ref{First extension} and \ref{Second extension}.
	\end{proof}
	
	\begin{proof}[Proof of Theorem \ref{Simplices and exposed faces}]
		Let $F \subseteq K$ be a closed face of $K$. By Lemma \ref{Closed faces are exposed}, the face $F$ is exposed, so let $\ell : K \to \mathbb{R}$ be a weak*-continuous affine functional such that
		\begin{align*}
			\ell(k)	& = 0	& (\forall k \in F) , \\
			\ell(k)	& < 0	& (\forall k \in K \setminus F) .
		\end{align*}
		Set
		$$V = \left\{ c_1 \phi_1 - c_2 \phi_2 : c_1, c_2 \in \mathbb{R}_{\geq 0} ; \phi_1, \phi_2 \in K \right\} ,$$
		and let $\tilde{\ell} : V \to \mathbb{R}$ be a continuous linear extension of $\ell$ to $V$ whose existence is promised by Theorem \ref{Extension Theorem}. We can then extend $\tilde{\ell} : V \to \mathbb{R}$ to a weak*-continuous linear functional $\ell' : \mathfrak{R}^\natural \to \mathbb{R}$ \cite[Theorem 3.6]{PositiveOperators}. There thus exists some $x \in \mathbb{R}$ such that $\ell'(\phi) = \phi(x)$ for all $\phi \in \mathfrak{R}^\natural$ \cite[Theorem 5.2]{Baggett}. In particular, we have $\ell'(v) = v(x)$ for all $v \in V$. Therefore $F = K_\mathrm{max}(x)$.
		
		The converse is contained in Proposition \ref{Nonempty maximizing states}.
	\end{proof}
	
	In particular, we can recover the following corollary.
	
	\begin{Cor}\label{Uniquely maximizing states}
		If $\phi \in \partial_e K$, then there exists $x \in \mathfrak{R}$ such that $\phi$ is uniquely $(x \vert K)$-maximizing, i.e. such that $\{ \phi \} =  K_\mathrm{max}(x)$.
	\end{Cor}
	
	\begin{proof}
		The singleton $\{\phi\}$ is a closed face, and by Lemma \ref{Closed faces are exposed} is therefore an exposed face. Apply Theorem \ref{Simplices and exposed faces}.
	\end{proof}
	
	We have developed the language of ergodic optimization here in a somewhat novel way, where we speak not of $x$-maximizing states \emph{simpliciter} -as one would speak of $f$-maximizing measures in the commutative theory we draw inspiration from- but of a state that is maximizing relative to a compact convex subset $K$ of $\mathcal{S}^G$, especially a compact simplex $K$. This means we can consider ergodic optimization problems over different \emph{types} of states.
	
	Since our machinery works best for cases where $K$ is a simplex, we will conclude this section by describing some situations where $\mathcal{S}^G$ is a compact metrizable simplex.
	
	For each $\phi \in \mathcal{S}^G$, let $\pi_\phi : \mathfrak{A} \to \mathscr{B} (\mathscr{H}_\phi)$ be the GNS representation corresponding to $\phi$. Define a unitary representation $u_\phi : G \to \mathbb{U}(\mathscr{H}_\phi)$ of $G$ by
	$$u_\phi(g) \pi_\phi(a) = \pi_\phi \left( \Theta_{g^{-1}} (a) \right) ,$$
	extending this from $\pi_\phi(\mathfrak{A})$ to $\mathscr{H}_\phi$. Set
	$$E_\phi = \left\{ v \in \mathscr{H}_\phi : u_\phi(v) = v \textrm{ for all } g \in G \right\} .$$
	Let $P_\phi : \mathscr{H}_\phi \twoheadrightarrow E_\phi$ be the orthogonal projection (in the functional-analytic sense) of $\mathscr{H}_\phi$ onto $E_\phi$. We call the C*-dynamical system $(\mathfrak{A}, G, \Theta)$ a \emph{$G$-abelian} system if for every $\phi \in \mathcal{S}^G$, the family of operators $\left\{ P_\phi \pi_\phi(a) P_\phi \in \mathscr{B}(\mathscr{H}_\phi) : a \in \mathfrak{A} \right\}$ is mutually commutative.
	
	We record here a handful of germane facts about $G$-abelian systems.
	
	\begin{Prop}
		If $(\mathfrak{A}, G, \Theta)$ is $G$-abelian, then $\mathcal{S}^G$ is a simplex.
	\end{Prop}
	
	\begin{proof}
		See \cite[Theorem 3.1.14]{Sakai}.
	\end{proof}
	
	\begin{Def}
		We call a system $(\mathfrak{A}, G, \Theta)$ \emph{asymptotically abelian} if there exists a sequence $(g_n)_{n = 1}^\infty$ in $G$ such that
		$$\left[ \Theta_{g_n} a, b \right] \stackrel{n \to \infty}{\to} 0 $$
		for all $a, b \in \mathfrak{A}$, where $\left[ \cdot , \cdot \right]$ is the Lie bracket $[x, y] = xy - yx$ on $\mathfrak{A}$.
	\end{Def}
	
	\begin{Prop}
		If $(\mathfrak{A}, G, \Theta)$ is asymptotically abelian, then it is also $G$-abelian.
	\end{Prop}
	
	\begin{proof}
		See \cite[Proposition 3.1.16]{Sakai}.
	\end{proof}
	
	\section{Unique ergodicity and gauges: the singly generated setting}\label{Singly generated UE}
	
	So far we have spoken about C*-dynamical systems, a noncommutative analog of a topological dynamical systems. But just as classical ergodic theory is often interested in the interplay between topological dynamical systems and the measure-theoretic dynamical systems they can be realized in, we are interested in questions about the interplay between C*-dynamical systems and the non-commutative measure-theoretic dynamical systems they can be realized in. To make this more precise, we must define the notion of a W*-dynamical system.
	
	A \emph{W*-probability space} is a pair $(\mathfrak{M}, \rho)$ consisting of a von Neumann algebra $\mathfrak{M}$ and a faithful tracial normal state $\rho$ on $\mathfrak{M}$. An \emph{automorphism} of a W*-probability space $(\mathfrak{M}, \rho)$ is a *-automorphism $T : \mathfrak{M} \to \mathfrak{M}$ such that $\rho \circ T = \rho$, i.e. an automorphism of $\mathfrak{M}$ which respects $\rho$. A \emph{W*-dynamical system} is a quadruple $(\mathfrak{M}, \rho, G, \Xi)$, where $(\mathfrak{M}, \rho)$ is a W*-probability space, and $\Xi : G \to \operatorname{Aut}(\mathfrak{M}, \rho)$ is an action of a discrete topological group $G$ (called the \emph{phase group}) on $\mathfrak{M}$ by $\rho$-preserving automorphisms of $\mathfrak{M}$, i.e. such that $\rho(\Xi_g x) = \rho(x)$ for all $g \in G, x \in \mathfrak{M}$. Importantly, if $(\mathfrak{M}, \rho, G, \Xi)$ is a W*-dynamical system, then $(\mathfrak{M}, G, \Xi)$ is automatically a W*-dynamical system.
	
	\begin{Rmk}
		Elsewhere in the literature, the term W*-dynamical system is typically used to refer to a more general construction, where the group $G$ is assumed to satisfy some topological conditions, and the action is assumed to be continuous in the strong operator topology, e.g. \cite{NoncommutativeJoinings}. Other authors use a yet more general definition, e.g. \cite[III.3.2]{Blackadar}. Since we are only interested in actions of discrete groups, we adopt a narrower definition. Our introduction of the W*-probability space is technically superfluous, but allows us to emphasize when certain properties of a W*-dynamical system $(\mathfrak{M}, \rho, G, \Xi)$ are intrinsic to $(\mathfrak{M}, \rho)$ instead of the action placed upon it. In particular, certain regularity properties arise when we assume that the Hilbert space $\mathcal{L}^2(\mathfrak{M}, \rho)$ (defined below) is separable, analogous to some of the regularity properties that arise in classical ergodic theory when we assume that the underlying probability space is standard.
	\end{Rmk}
	
	\begin{Def}
	Given a W*-probability space, we define $\mathcal{L}^2(\mathfrak{M}, \rho)$ to be the Hilbert space defined by completing $\mathfrak{M}$ with respect to the inner product $\left< x , y \right>_\rho = \rho \left( y^* x \right)$, i.e. the Hilbert space associated with the faithful GNS representation of $\mathfrak{M}$ induced by $\rho$.
	\end{Def}
	
	Finally, we introduce the notion of a C*-model, intending to generalize the notion of a topological model from classical ergodic theory to this noncommutative setting.
	
	\begin{Def}
		Let $(\mathfrak{M}, \rho, G, \Xi)$ be a W*-dynamical system. A \emph{C*-model} of $(\mathfrak{M}, \rho, G, \Xi)$ is a quadruple $(\mathfrak{A}, G, \Theta; \iota)$ consisting of a C*-dynamical system $(\mathfrak{A}, G, \Theta)$ and a *-homomorphism $\iota : \mathfrak{A} \to \mathfrak{M}$ such that
		\begin{enumerate}[label=(\alph*)]
			\item $\iota(\mathfrak{A})$ is dense in the weak operator topology of $\mathfrak{M}$,
			\item $\Xi_g \left( \iota(\mathfrak{A}) \right) = \iota(\mathfrak{A})$ for all $g \in G$, and
			\item $\Xi_g \circ \iota = \iota \circ \Theta_g$ for all $g \in G$.
		\end{enumerate}
		We call the C*-model $(\mathfrak{A}, G, \Theta; \iota)$ \emph{faithful} if $\iota$ is also injective.
	\end{Def}
	
	Before continuing, we want to remark that we can turn any C*-model into a faithful C*-model through a quotienting process. If $\iota$ was not injective, then we could instead consider $\tilde{\iota} : \mathfrak{A} / \ker \iota \hookrightarrow \mathfrak{M}$. In the case where $\mathfrak{A}$ is commutative, this quotienting process corresponds (via the Gelfand-Naimark Theorem) to taking a measure-theoretic dynamical system and restricting to the support of the resident probability measure. To see this, let $\mathfrak{A} = C(X)$, where $X$ is a compact metrizable topological space, and let $\mathfrak{M} = L^\infty(X, \mu)$ for some Borel probability measure $\mu$. Let $\iota : C(X) \to L^\infty(X, \mu)$ be the (not necessarily injective) map that maps a continuous function on $X$ to its equivalence class in $L^\infty(X, \mu)$. It can be seen that $f \in \ker \iota$ if and only if the open set $\left\{ x \in X : f(x) \neq 0 \right\}$ is of measure $0$, or equivalently if $f \vert_{\operatorname{supp}(\mu)} = 0$, and in particular that $\iota$ is injective if and only if $\mu$ is \emph{strictly positive} (i.e. $\mu$ assigns positive measure to all nonempty open sets). As such, we can identify $C(X) / \ker \iota$ with $C(\operatorname{supp}(\mu))$. Let $Y = \operatorname{supp}(\mu)$ denote the support of $\mu$ on $X$, and let $\pi : C(X) \twoheadrightarrow C(Y)$ be the quotient map (which corresponds to a restriction from $X$ to $Y$, i.e. $\pi f = f \vert_Y$). Then algebraically, we have a diagram
	$$
	\begin{tikzcd}
		C(X) \arrow[r, two heads, "\pi"] \arrow[rd, "\iota"]	& C(Y) \arrow[d, hook, dotted, "\tilde{\iota}"] \\
		& L^\infty(X, \mu)
	\end{tikzcd}
	$$
	So in the commutative case, we can make $\iota : C(X) \to L^\infty(X, \mu)$ injective by looking at $\tilde{\iota} : C(Y) \to L^\infty(Y, \mu) \cong L^\infty(X, \mu)$, i.e. by using the support $Y$ to model $(Y, \mu) \cong (X, \mu)$.
	
	Importantly, so long as $\mathcal{L}^2(\mathfrak{M}, \rho)$ is separable, any W*-dynamical system $(\mathfrak{M}, \rho, G, \Xi)$ will admit a faithful separable C*-model. To construct such a C*-model, it suffices to take some separable C*-subalgebra $\mathfrak{B} \subseteq \mathfrak{M}$ which is dense in $\mathfrak{M}$ with respect to the weak operator topology, then let $\mathfrak{A}$ be the norm-closure of the span of $\bigcup_{g \in G} \left( \Xi_g \mathfrak{B} \right)$. We then define $\Theta_g = \Xi_g \vert_{\mathfrak{A}}$ and let $\iota : \mathfrak{A} \hookrightarrow \mathfrak{M}$ be the inclusion map.
	
	One last important concept in this section and the next will be unique ergodicity. A C*-dynamical system $(\mathfrak{A}, G, \Theta)$ is called \emph{uniquely ergodic} if $\mathcal{S}^G$ is a singleton. As in the commutative setting, unique ergodicity can be equivalently characterized in terms of convergence properties of ergodic averages. To our knowledge, the strongest such characterization of unique ergodicity for singly generated C*-dynamical systems can be found in \cite[Theorem 3.2]{AbadieDykema}, which describes relative unique ergodicity. This characterization was then generalized to characterize relative unique ergodicity in amenable C*-dynamical systems in \cite[Theorem 5.2]{DuvenhageStroeh}; however, in Theorem \ref{Classical unique ergodicity for amenable groups}, we provide a characterization of uniquely ergodic C*-dynamical systems in terms of ergodic averages that is not encompassed by \cite[Theorem 5.2]{DuvenhageStroeh}.
	
	For the duration of this section, we assume that $(\mathfrak{M}, \rho, \mathbb{Z}, \Xi)$ is a W*-dynamical system with $\mathcal{L}^2(\mathfrak{M}, \rho)$ separable, and that $(\mathfrak{A}, \mathbb{Z}, \Theta)$ is a C*-dynamical system such that $\mathfrak{A}$ is separable.
	
	Given a C*-dynamical system $(\mathfrak{A}, \mathbb{Z}, \Theta)$, let $a \in \mathfrak{A}$ be a positive element. We define the \emph{gauge} of $a$ to be
	$$\Gamma(a) : = \lim_{k \to \infty} \frac{1}{k} \left\| \sum_{j = 0}^{k - 1} \Theta_j a \right\| .$$
	To prove this limit exists, it suffices to observe that the sequence $\left( \left\|\sum_{j = 0}^{k - 1} \Theta_j a \right\| \right)_{k = 1}^\infty$ is subadditive, since
	\begin{align*}
		\left\| \sum_{j = 0}^{k + \ell - 1} \Theta_j a \right\|	& \leq \left\| \sum_{j = 0}^{k - 1} \Theta_j a \right\| + \left\| \sum_{j = k}^{k + \ell - 1} \Theta_j a \right\| \\
		& = \left\| \sum_{j = 0}^{k - 1} \Theta_j a \right\| + \left\| \Theta_k \sum_{j = 0}^{\ell - 1} \Theta_j a \right\| \\
		& = \left\| \sum_{j = 0}^{k - 1} \Theta_j a \right\| + \left\| \sum_{j = 0}^{\ell - 1} \Theta_j a \right\| .
	\end{align*}
	Therefore, by the Subadditivity Lemma, the sequence $\left( \frac{1}{k} \left\|\sum_{j = 0}^{k - 1} \Theta_j a \right\| \right)_{k = 1}^\infty$ converges, and we have the equality
	$$\lim_{k \to \infty} \frac{1}{k} \left\| \sum_{j = 0}^{k - 1} \Theta_j a \right\| = \inf_{k \in \mathbb{N}} \frac{1}{k} \left\| \sum_{j = 0}^{k - 1} \Theta_j a \right\| .$$
	
	We have the following characterization of $\Gamma$ in the language of ergodic optimization.
	
	\begin{Thm}\label{Singly generated gauge}
	Let $(\mathfrak{A}, \mathbb{Z}, \Theta)$ be a C*-dynamical system. Then if $a \in \mathfrak{A}$ is a positive element, then $\Gamma(a) = m \left( a \vert \mathcal{S}^G \right)$.
	\end{Thm}

	\begin{proof}
	For each $k \in \mathbb{N}$, choose a state $\sigma_k$ on $\mathfrak{A}$ such that
	$$\sigma_k \left( \frac{1}{k} \sum_{j = 0}^{k - 1} \Theta_j a \right) = \left\| \frac{1}{k} \sum_{j = 0}^{k - 1} \Theta_j a \right\| .$$
	Let $\omega_k = \frac{1}{k} \sum_{j = 0}^{k - 1} \sigma_k \circ \Theta_j$, so
	\begin{align*}
		\omega_k(x)	& = \frac{1}{k} \sum_{j = 0}^{k - 1} \sigma_k \left( \Theta_j x \right)  \\
		& = \sigma_k \left( \frac{1}{k} \sum_{j = 0}^{k - 1} \Theta_j x \right) , \\
		\omega_k(a)	& = \sigma_k \left( \frac{1}{k} \sum_{j = 0}^{k - 1} \Theta_j a \right) \\
		& = \left\| \frac{1}{k} \sum_{j = 0}^{k - 1} \Theta_j a \right\| .
	\end{align*}
	Let $\omega \in \mathcal{S}$ be a weak*-limit point of $\left( \omega_k : k \in \mathbb{N} \right)$, and let $k_1 < k_2 < \cdots$ be a subsequence such that $\omega_{k_n} \stackrel{n \to \infty}{\to} \omega$ in the weak*-topology. We claim that $\omega$ is $\Theta$-invariant. This follows because if $x \in \mathfrak{A}$, then
	\begin{align*}
		\left| \omega(x - \Theta_1 x) \right|	& = \left| \lim_{n \to \infty} \omega_{k_n} (x - \Theta_1 x) \right| \\
		& = \lim_{n \to \infty} \left| \sigma_{k_n} \left( \left( \frac{1}{k_n}\sum_{j = 0}^{k_n - 1} \Theta_j x \right) - \left( \frac{1}{k_n}\sum_{j = 0}^{k_n - 1} \Theta_j \Theta_1 x \right) \right) \right| \\
		& = \lim_{n \to \infty} \frac{1}{k_n} \left| \sigma_{k_n} \left( x - \Theta_{k_n} x \right) \right| \\
		& \leq \lim_{n \to \infty} \frac{2 \| x\|}{k_n} \\
		& = 0 .
	\end{align*}
	Therefore $\omega(a) = \Gamma(a)$, and $\omega$ is a $\Theta$-invariant state on $\mathfrak{A}$, so $$\omega(a) = \phi(a) \leq m \left( a \vert \mathcal{S}^\mathbb{Z} \right) .$$
	
	Now, we prove the opposite inequality. Let $\phi \in \mathcal{S}^\mathbb{Z}$. Then
	\begin{align*}
		\phi(a)	& = \phi \left( \operatorname{Avg}_k a \right) \\
		& \leq \left\| \operatorname{Avg}_k a \right\| \\
		& = \frac{1}{k} \left\| \sum_{j = 0}^{k - 1} \Theta_j a \right\| \\
		& = \frac{1}{k} \left\| \sum_{j = 0}^{k - 1} \Theta_j a \right\|	& (\forall k \in \mathbb{N}) \\
		\Rightarrow \phi(a)	& \leq \inf_{k \in \mathbb{N}} \frac{1}{k} \left\| \sum_{j = 0}^{k - 1} \Theta_j a \right\| \\
		& = \Gamma(a) \\
		\Rightarrow \sup_{\psi \in \mathcal{S}^\mathbb{Z}} \psi(a)	& \leq \Gamma(a) .
	\end{align*}
	Therefore
	$$m \left( a \vert \mathcal{S}^\mathbb{Z} \right) = \sup_{\psi \in \mathcal{S}^\mathbb{Z}} \psi(a) \leq \Gamma(a) .$$
	
	This establishes the identity.
	\end{proof}
	
	\begin{Cor}\label{Gamma estimate}
		Let $(\mathfrak{M}, \rho, \mathbb{Z}, \Xi)$ be a W*-dynamical system, and let $(\mathfrak{A}, \mathbb{Z}, \Theta; \iota)$ be a C*-model of $(\mathfrak{M}, \rho, \mathbb{Z}, \Xi)$. If $a \in \mathfrak{A}$ is a positive element, then
		$$\Gamma(\iota(a)) = m \left( a \vert \operatorname{Ann}(\ker \iota) \right) .$$
	\end{Cor}
	
	\begin{proof}
	Write $\tilde{\mathfrak{A}} = \iota(\mathfrak{A}) \subseteq \mathfrak{M}$, and let $\tilde{\Theta} : \mathbb{Z} \to \operatorname{Aut} \left( \tilde{\mathfrak{A}} \right)$ be the action $\tilde{ \Theta }_n = \Xi_n \vert_{ \tilde{\mathfrak{A}} }$ obtained by restricting $\Xi$ to $\tilde{\mathfrak{A}}$. Write $\tilde{\mathcal{S}}^\mathbb{Z}$ for the space of $\tilde{\Theta}$-invariant states on $\tilde{\mathfrak{A}}$.
	
	We can write $\Gamma_{\mathfrak{M}}(\iota(a)) = \Gamma_{\tilde{\mathfrak{A}}}(\iota(a))$. By Theorem \ref{Singly generated gauge}, we know that $\Gamma_{ \tilde{\mathfrak{A}} }(\iota(a)) = m \left( \iota(a) \vert \tilde{\mathcal{S}}^\mathbb{Z} \right)$, and by Theorem \ref{Ergodic optimization through *-homomorphisms}, we know that $m \left( \iota(a) \vert \tilde{\mathcal{S}}^\mathbb{Z} \right) = m \left( a \vert \operatorname{Ann}(\ker \iota) \right)$.
	\end{proof}
	
	\begin{Rmk}
	Corollary \ref{Gamma estimate} can be regarded as an operator-algebraic extension of Lemma 2.3 from \cite{Assani-Young}. The assumption that $(\mathfrak{A}, G, \Theta; \iota)$ is faithful can be understood as analogous to the assumption of strict positivity in that paper.
	\end{Rmk}
	
	This $\Gamma$ functional provides an alternative characterization of unique ergodicity, at least under some additional Choquet-theoretic hypotheses.
	
	\begin{Thm}
		Let $(\mathfrak{M}, \rho, \mathbb{Z}, \Xi)$ be a W*-dynamical system, and let $(\mathfrak{A}, \mathbb{Z}, \Theta; \iota)$ be a faithful C*-model of $(\mathfrak{M}, \rho, \mathbb{Z}, \Xi)$. Then the following conditions are related by the implications (i)$\iff$(ii)$\Rightarrow$(iii).
		\begin{enumerate}[label=(\roman*)]
			\item The C*-dynamical system $(\mathfrak{A}, \mathbb{Z}, \Theta)$ is uniquely ergodic.
			\item The C*-dynamical system $(\mathfrak{A}, \mathbb{Z}, \Theta)$ is strictly ergodic.
			\item $\Gamma(\iota(a)) = \rho(\iota(a))$ for all positive $a \in \mathfrak{A}$.
		\end{enumerate}
		Further, if $\mathcal{S}^\mathbb{Z}$ is a simplex, then (iii)$\Rightarrow$(i).
	\end{Thm}
	
	\begin{proof}
		(i)$\Rightarrow$(ii) Suppose that $(\mathfrak{A}, \mathbb{Z}, \Theta)$ is uniquely ergodic. Then $\rho \circ \iota$ is an invariant state on $\mathfrak{A}$, so it follows that $\rho \circ \iota$ is the unique invariant state on $\mathfrak{A}$. But $\rho \circ \iota$ is also a faithful state on $\mathfrak{A}$, so it follows that $(\mathfrak{A}, \mathbb{Z}, \Theta)$ is strictly ergodic.
		
		(ii)$\Rightarrow$(i) Trivial.
		
		(i)$\Rightarrow$(iii) Suppose that $(\mathfrak{A}, \mathbb{Z}, \Theta)$ is uniquely ergodic, and let $a \in \mathfrak{A}$ be positive. Let $\phi$ be a $\mathcal{S}^\mathbb{Z}$-maximizing state for $a$. Then $\phi = \rho \circ \iota$, since both $\phi$ and $\rho \circ \iota$ are invariant states on $\mathfrak{A}$, and $(\mathfrak{A}, \mathbb{Z}, \Theta)$ is uniquely ergodic. Thus $\phi = \rho \circ \iota$, so $\Gamma(\iota(a)) = \phi(a) = \rho(\iota(a))$.
		
		(iii)$\Rightarrow$(i) Suppose that $\mathcal{S}^\mathbb{Z}$ is a simplex, but that $(\mathfrak{A}, \mathbb{Z}, \Theta)$ is \emph{not} uniquely ergodic. By the Krein-Milman Theorem, there exists an extreme point $\phi \in \mathcal{S}^\mathbb{Z}$ of $\mathcal{S}^\mathbb{Z}$ different from $\rho \circ \iota$. Then by Corollary \ref{Uniquely maximizing states}, there exists $a \in \mathfrak{A}$ self-adjoint such that $\{ \phi \} = \mathcal{S}_\mathrm{max}^\mathbb{Z}(a)$. We can assume that $a$ is positive, since otherwise we could replace $a$ with $a + r$ for a sufficiently large positive real number $r > 0$, and $\mathcal{S}_\mathrm{max}^\mathbb{Z}(a) = \mathcal{S}_\mathrm{max}^\mathbb{Z}(a + r)$. Then $\Gamma(\iota(a)) = \phi(a)$. But by the assumption that $\phi$ is uniquely $\left( a \vert \mathcal{S}^\mathbb{Z} \right)$-maximizing, it follows that $\rho(\iota(a)) < \phi(a)$. Therefore $\Gamma(\iota(a)) \neq \rho(\iota(a))$, meaning that (iii) does not attain. Thus $\neg$(i)$\Rightarrow \neg$(iii).
	\end{proof}
	
	\section{Unique ergodicity and gauge: the amenable setting}\label{Amenable UE}
	
	For the duration of this section, we assume that $(\mathfrak{M}, \rho, G, \Xi)$ is a W*-dynamical system with $\mathcal{L}^2(\mathfrak{M}, \rho)$ separable. Assume further that $(\mathfrak{A}, G, \Theta)$ is a C*-dynamical system such that $\mathfrak{A}$ is separable, and that $G$ is amenable. It follows from Corollary \ref{Tracial K-B} that $\mathcal{S}^G \neq \emptyset$.
	
	In this section, we expand upon some of the ideas presented in Section \ref{Singly generated UE}, generalizing from the case of actions of $\mathbb{Z}$ to actions of a countable discrete amenable group $G$. We separate these two sections because our treatment of the more general amenable setting has some additional nuances to it.
	
	For an arbitrary nonempty finite subset $F$ of $G$, set
	$$\operatorname{Avg}_F x : = \frac{1}{|F|} \sum_{g \in F} \Theta_g x .$$

	Our first result of this section is a generalization of a classical result from ergodic theory regarding unique ergodicity, which is that a (singly generated) topological dynamical system is uniquely ergodic if and only if the averages of the continuous functions converge to a constant. This classical result is well-known, and can be found in many standard texts on ergodic theory, e.g. \cite[Thm 6.2.1]{DajaniDirksin}, \cite[Thm 10.6]{EisnerOperators}, \cite[Thm 5.17]{Walters}, but the earliest example of a result like this that we could find was \cite[5.3]{OxtobyErgodic}. Theorem \ref{Unique ergodicity equivalent statements} generalizes this classical result not only to the noncommutative setting, but to the setting where the phase group $G$ is amenable.
	
	We define the \emph{weak topology} on a C*-algebra $\mathfrak{A}$ to be the topology generated by the states on $\mathfrak{A}$, i.e.
	\begin{align*}
		x	& \mapsto \psi(x)	& (\psi \in \mathcal{S}) .
	\end{align*}
	In other words, the weak topology is the topology in which a net $(x_i)_i$ converges to $x$ if and only if $(\psi(x_i) )_i$ converges to $\psi(x)$ for every state $\psi$ on $\mathfrak{A}$. We say the net $(x_i)_i$ \emph{converges weakly} to $x$ if it converges in the weak topology.
	
	\begin{Thm}\label{Unique ergodicity equivalent statements}
		Let $(\mathfrak{A}, G, \Theta)$ be a C*-dynamical system. Then the following conditions are equivalent.
		\begin{enumerate}[label=(\roman*)]
			\item $(\mathfrak{A}, G, \Theta)$ is uniquely ergodic.
			\item There exists a left Følner sequences $(F_k)_{k = 1}^\infty$ for $G$ and a linear functional $\phi : \mathfrak{A} \to \mathbb{C}$ such that for all $x \in \mathfrak{A}$, the sequence $\left( \operatorname{Avg}_{F_k} x \right)_{k = 1}^\infty$ converges in norm to $\phi(x) 1 \in \mathbb{C} 1$.
			\item There exists a left Følner sequences $(F_k)_{k = 1}^\infty$ for $G$ and a linear functional $\phi : \mathfrak{A} \to \mathbb{C}$ such that for all $x \in \mathfrak{A}$, the sequence $\left( \operatorname{Avg}_{F_k} x \right)_{k = 1}^\infty$ converges weakly to $\phi(x) 1 \in \mathbb{C} 1$.
			\item There exists a state $\phi$ on $\mathfrak{A}$ such that for every right Følner sequence $(F_k)_{k = 1}^\infty$ for $G$, the sequence $\left( \operatorname{Avg}_{F_k} x \right)_{k = 1}^\infty$ converges in norm to $\phi(x) 1 \in \mathbb{C} 1$.
			\item There exists a state $\phi$ on $\mathfrak{A}$ such that for every right Følner sequence $(F_k)_{k = 1}^\infty$ for $G$, the sequence $\left( \operatorname{Avg}_{F_k} x \right)_{k = 1}^\infty$ converges weakly to $\phi(x) 1 \in \mathbb{C} 1$.
		\end{enumerate}
	\end{Thm}
	
	\begin{proof}
		Assume throughout that any $x \in \mathfrak{A}$ is nonzero.	
		
		(ii)$\Rightarrow$(iii) Obvious.
		
		(iv)$\Rightarrow$(v) Obvious.
		
		(iv)$\Rightarrow$(ii) Follows because a two-sided Følner sequence exists.
		
		(v)$\Rightarrow$(iii) Follows because a two-sided Følner sequence exists.
		
		(iii)$\Rightarrow$(i) Suppose that $\operatorname{Avg}_{F_k} x \to \phi(x) 1 \in \mathbb{C} 1$ weakly for all $x \in \mathfrak{A}$. We claim that $\phi$ is the unique invariant state of $(\mathfrak{A}, G, \Theta)$. First, we demonstrate that $\phi$ is $\Theta$-invariant. Fix $g_0 \in G$, and fix $\epsilon > 0$. Choose $K_1, K_2, K_3 \in \mathbb{N}$ such that
		\begin{align*}
			k	& \geq K_1	& \Rightarrow \left| \phi(\phi(x) 1) - \phi \left( \operatorname{Avg}_{F_k} x \right) \right|	& < \frac{\epsilon}{3} , \\
			k	& \geq K_2	& \Rightarrow \left| \phi(\Theta_{g_0} \phi(x) 1) - \phi (\Theta_{g_0} \operatorname{Avg}_{F_k} x ) \right|	& < \frac{\epsilon}{3} , \\
			k	& \geq K_3	& \Rightarrow \frac{|g_0 F_k \Delta F_k|}{|F_k|}	& < \frac{\epsilon}{3 \| x \|} .
		\end{align*}
		The $K_1, K_2$ exist because we know that in the weak topology, the functionals $\phi , \phi \circ \Theta_{g_0}$ are both continuous, and $K_3$ exists by the amenability of $G$. Let $K = \max \{K_1, K_2, K_3\}$. Then if $k \geq K$, then
		\begin{align*}
			\left| \phi( \Theta_{g_0} x) - \phi(x) \right|	& \leq \left| \phi(\Theta_{g_0} x) - \phi(\Theta_{g_0} \operatorname{Avg}_{F_k} x) \right| + \left| \phi(\Theta_{g_0} \operatorname{Avg}_{F_k} x) - \phi(\operatorname{Avg}_{F_k} x) \right| + \left| \phi(\operatorname{Avg}_{F_k} x) - \phi(x) \right| \\
			& \leq \frac{\epsilon}{3} + \left| \phi(\Theta_{g_0} \operatorname{Avg}_{F_k} x) - \phi(\operatorname{Avg}_{F_k} x) \right| + \frac{\epsilon}{3} \\
			& = \frac{2 \epsilon}{3} + \left| \phi(\Theta_{g_0} \operatorname{Avg}_{F_k} x) - \phi(\operatorname{Avg}_{F_k} x) \right| \\
			& = \frac{2 \epsilon}{3} + \left| \phi \left( \frac{1}{|F_k|}\left( \sum_{g \in F_k} \Theta_{g_0 g} x \right) - \left( \frac{1}{|F_k|} \sum_{g \in F_k} \Theta_g x \right) \right) \right| \\
			& = \frac{2 \epsilon}{3} + \left| \phi \left( \frac{1}{|F_k|}\left( \sum_{g \in g_0 F_k} \Theta_{g} x \right) - \left( \frac{1}{|F_k|} \sum_{g \in F_k} \Theta_g x \right) \right) \right| \\
			& = \frac{2 \epsilon}{3} + \left| \phi \left( \frac{1}{|F_k|}\left( \sum_{g \in g_0 F_k \setminus F_k} \Theta_{g} x \right) - \left( \frac{1}{|F_k|} \sum_{g \in F_k \setminus g_0 F_k} \Theta_g x \right) \right) \right| \\
			& \leq \frac{2 \epsilon}{3} + \left| \phi \left( \frac{1}{|F_k|}\sum_{g \in g_0 F_k \setminus F_k} \Theta_{g} x \right) \right| + \left| \phi \left( \frac{1}{|F_k|} \sum_{g \in F_k \setminus g_0 F_k} \Theta_g x \right) \right| \\
			& < \frac{2 \epsilon}{3} + \frac{|g_0 F_k \Delta F_k|}{|F_k|} \| x \| \\
			& = \epsilon .
		\end{align*}
		Therefore $\phi$ is $\Theta$-invariant. To see that it is positive, it suffices to observe that $x \geq 0 \Rightarrow \operatorname{Avg}_{F_k}$, meaning that $\phi(x) = \lim_{k \to \infty} \phi(\operatorname{Avg}_{F_k} x) \geq 0$. To see that $\phi(1) = 1$, we just observe that $\operatorname{Avg}_{F_k} 1 = 1$ for all $k \in \mathbb{N}$.
		
		Now we show that $\phi$ is the \emph{unique} $\Theta$-invariant state. Let $\psi$ be any invariant state. Then
		\begin{align*}
			\psi(x)	& = \psi(\operatorname{Avg}_{F_k} x) \\
			& \stackrel{k \to \infty}{\to} \psi(\phi(x) 1) \\
			& = \phi(x) \psi(1) \\
			& = \phi(x) .
		\end{align*}
		Therefore $\psi = \phi$, and so $(\mathfrak{A}, G, \Theta)$ is uniquely ergodic.
		
		(i)$\Rightarrow$(iv) Fix a right Følner sequence $(F_k)_{k = 1}^\infty$, and assume for contradiction that $(\mathfrak{A}, G, \Theta)$ is uniquely ergodic with $\Theta$-invariant state $\phi$, but that there exists $x \in \mathfrak{A}$ such that $\left( \operatorname{Avg}_{F_k} x \right)_{k = 1}^\infty$ does \emph{not} converge in norm to a scalar, and in particular does not converge in norm to $\phi(x) 1$. Since we can decompose $x$ into its real and imaginary parts, we can assume that $x \in \mathfrak{A}_\mathrm{sa}$. Fix $\epsilon_0 > 0$ for which there exists an infinite sequence $k_1 < k_2 < \cdots$ such that $\left\| \operatorname{Avg}_{F_{k_n}} x - \phi(x) 1 \right\| \geq \epsilon_0$. Then for each $n \in \mathbb{N}$ exists a state $\psi_n$ on $\mathfrak{A}$ such that $\left| \psi_n \left( \operatorname{Avg}_{F_{k_n}} x - \phi(x) 1 \right) \right| = \left\| \operatorname{Avg}_{F_{k_n}} x - \phi(x) 1 \right\|$.
		
		Set
		$$\omega_n = \psi_n \circ \operatorname{Avg}_{F_{k_n}} ,$$
		so $\omega_n (x - \phi(x) 1) = \psi_n \left( \operatorname{Avg}_{F_{k_n}} x - \phi(x) 1 \right)$. Then $(\omega_n)_{n = 1}^\infty$ has a subsequence, call it $(\omega_{n_j})_{j = 1}^\infty$ which converges in the weak*-topology to some $\omega$. This $\omega$ is also a state on $\mathfrak{A}$, and we claim it is $\Theta$-invariant. Fix $g_0 \in G, y \in \mathfrak{A}$. Then
		\begin{align*}
			\left| \omega(\Theta_{g_0} y) - \omega(y) \right|	& = \lim_{j \to \infty} \left| \omega_{n_j}(\Theta_{g_0} y) - \omega_{n_j}(y) \right| \\
			& = \lim_{j \to \infty} \left| \psi_{n_j} \left( \Theta_{g_0} \operatorname{Avg}_{F_{k_{n_j}}} y \right) - \psi_{n_j} \left( \operatorname{Avg}_{F_{k_{n_j}}} y \right) \right| \\
			& = \lim_{j \to \infty} \frac{1}{|F_{k_{n_j}}|}\left| \psi \left( \left( \sum_{g \in F_{k_{n_j}} g_0} \Theta_g y \right) - \left( \sum_{g \in F_{k_{n_j}}} \Theta_g y \right) \right) \right| \\
			& \leq \limsup_{j \to \infty} \frac{ \left|F_{k_{n_j}} g_0 \Delta F_{k_{n_j}} \right|}{|F_{k_{n_j}}|} \| y \| \\
			& = 0 .
		\end{align*}
		Therefore $\omega$ is $\Theta$-invariant. But $\omega \neq \phi$, since
		\begin{align*}
			\left| \omega(x) - \phi(x) \right|	& = \lim_{j \to \infty} \left| \omega_{n_j}(x) - \phi(x) \right| \\
			& = \lim_{j \to \infty} \left| \omega_{n_j}(x - \phi(x) 1) \right| \\
			& = \lim_{j \to \infty} \left| \psi_{n_j}(\operatorname{Avg}_{F_{k_{n_j}}} x - \phi(x) 1) \right| \\
			& = \lim_{j \to \infty} \left\| \operatorname{Avg}_{F_{k_{n_j}}} x - \phi(x) 1 \right\| \\
			& \geq \epsilon_0 .
		\end{align*}
		This contradicts $(\mathfrak{A}, G, \Theta)$ being uniquely ergodic.
	\end{proof}
	
	\begin{Rmk}
		Although \cite[Theorem 5.2]{DuvenhageStroeh} describes conditions under which unique ergodicity of an action of an amenable group on a C*-algebra can be related to the convergence of ergodic averages, that result is not a direct generalization of our Theorem \ref{Unique ergodicity equivalent statements}.
	\end{Rmk}
	
	To the best of our knowledge, the following commutative analog of Theorem \ref{Unique ergodicity equivalent statements} has not been stated in the literature, so we treat it here as a corollary to Theorem \ref{Unique ergodicity equivalent statements}.
	
	\begin{Thm}\label{Classical unique ergodicity for amenable groups}
		Let $(X, G, U)$ be a topological dynamical system, where $G$ is amenable. For any nonempty finite $F \subseteq G, f \in C(X)$, let us abuse notation and write $\operatorname{Avg}_F f = \frac{1}{|F|} \sum_{g \in F} f \circ U_g$. Then the following conditions are equivalent.
		\begin{enumerate}[label=(\roman*)]
			\item $(X, G, U)$ is uniquely ergodic.
			\item There exists a left Følner sequences $(F_k)_{k = 1}^\infty$ for $G$ and a Borel probability measure $\mu$ on $X$ such that for all $f \in C(X)$, the sequence $\left( \operatorname{Avg}_{F_k} f \right)_{k = 1}^\infty$ converges uniformly to $\int f \mathrm{d} \mu$.
			\item There exists a left Følner sequences $(F_k)_{k = 1}^\infty$ for $G$ and a Borel probability measure $\mu$ on $X$ such that for all $f \in C(X)$, the sequence $\left( \operatorname{Avg}_{F_k} f \right)_{k = 1}^\infty$ converges pointwise to $\int f \mathrm{d} \mu$.
			\item There exists a Borel probability measure $\mu$ on $X$ such that for every right Følner sequence $(F_k)_{k = 1}^\infty$ for $G$, and for every $f \in C(X)$, the sequence $\left( \operatorname{Avg}_{F_k} f \right)_{k = 1}^\infty$ converges uniformly to $\int f \mathrm{d} \mu$.
			\item There exists a Borel probability measure $\mu$ on $X$ such that for every right Følner sequence $(F_k)_{k = 1}^\infty$ for $G$, and for every $f \in C(X)$, the sequence $\left( \operatorname{Avg}_{F_k} f \right)_{k = 1}^\infty$ converges pointwise to $\int f \mathrm{d} \mu$.
		\end{enumerate}
	\end{Thm}
	
	\begin{proof}
		This is essentially a corollary of Theorem \ref{Unique ergodicity equivalent statements}, since it is just the classical translation of that theorem in the case where $\mathfrak{A}$ is abelian. The only nontrivial part of translating that theorem to this one is to see that weak convergence of $\left(\operatorname{Avg}_{F_k} f\right)_{k = 1}^\infty$ to $\left( \int f \mathrm{d} \mu \right) 1$ is equivalent to the pointwise convergence of $\left( \operatorname{Avg}_{F_k} f \right)_{k = 1}^\infty$ to the constant function $\int f \mathrm{d} \mu$.
		
		First, suppose that $\left( \operatorname{Avg}_{F_k} f \right)_{k = 1}^\infty$ converges weakly to $\int f \mathrm{d} \mu$. Consider the evaluation states $\operatorname{ev}_y : f \mapsto f(y)$, where $y \in X$. Then $\operatorname{ev}_y \left(\operatorname{Avg}_{F_k} f \right) = \left( \operatorname{Avg}_{F_k} f \right) (y) \to \operatorname{ev}_ y \left(\int f \mathrm{d} \mu \right) = \int f \mathrm{d} \mu$. Therefore, the sequence converges pointwise to $\int f \mathrm{d} \mu$.
		
		Now, suppose that $\left( \operatorname{Avg}_{F_k} f \right)_{k = 1}^\infty$ converges pointwise to $\int f \mathrm{d} \mu$. Any state $\phi$ on $C(X)$ can be realized as $\phi(f) = \int f \mathrm{d} \nu$ for some Borel probability measure $\nu$ on $X$. The sequence $(\operatorname{Avg}_{F_k} f)_{k = 1}^\infty$ is uniformly bounded in magnitude by $\| f \|$, which is integrable with respect to any Borel probability measure, and so we can apply the Dominated Convergence Theorem to conclude that
		$$\int \operatorname{Avg}_{F_k} f \mathrm{d} \nu \to \int \left( \int f \mathrm{d} \mu \right) \mathrm{d} \nu = \int f \mathrm{d} \mu .$$
		In other words, $\phi(\operatorname{Avg}_{F_k} f) \to \int f \mathrm{d} \mu = \phi \left( \int f \mathrm{d} \mu \right)$ for any state $\phi$ on $C(X)$, so $\operatorname{Avg}_{F_k} f \to \int f \mathrm{d} \mu$ weakly.
	\end{proof}
	
	In order to develop the gauge machinery from the previous section in the context of actions of amenable groups, we will need to use slightly different techniques, since we do not have access to the Subadditivity Lemma. The main results of the remainder of this section can be summarized as follows.
	\begin{Main results}
		Let $\mathcal{F} = (F_k)_{k = 1}^\infty$ be a right Følner sequence.
		\begin{enumerate}[label=(\alph*)]
			\item Let $(\mathfrak{A}, G, \Theta)$ be a C*-dynamical system, and let $\mathcal{F} = (F_k)_{k = 1}^\infty$ be a right Følner sequence for $G$. Then if $a \in \mathfrak{A}$ is a positive element, then the sequence
			$\left( \left\| \frac{1}{|F_k|} \sum_{g \in F_k} \Theta_g a \right\| \right)_{k = 1}^\infty$ converges to $m \left( a \vert \mathcal{S}^G \right)$.
			\item Let $(\mathfrak{A}, G, \Theta; \iota)$ be a faithful C*-model of $(\mathfrak{M}, \rho, G, \Xi)$. Then the following conditions are related by the implications (i)$\iff$(ii)$\Rightarrow$(iii).
			\begin{enumerate}[label=(\roman*)]
				\item The C*-dynamical system $(\mathfrak{A}, G, \Theta)$ is uniquely ergodic.
				\item The C*-dynamical system $(\mathfrak{A}, G, \Theta)$ is strictly ergodic.
				\item $\Gamma(\iota(a)) = \rho(\iota(a))$ for all positive $a \in \mathfrak{A}$.
			\end{enumerate}
			Further, if $\mathcal{S}^G$ is a simplex, then (iii)$\Rightarrow$(i).
		\end{enumerate}
	\end{Main results}

	\begin{Thm}\label{Gauge exists for C*-dynamical systems}
		Let $(\mathfrak{A}, G, \Theta)$ be a C*-dynamical system, and let $\mathcal{F} = (F_k)_{k = 1}^\infty$ be a right Følner sequence for $G$. Then if $a \in \mathfrak{A}$ is a positive element, then the sequence
		$\left( \left\| \frac{1}{|F_k|} \sum_{g \in F_k} \Theta_g a \right\| \right)_{k = 1}^\infty$ converges to $m \left( a \vert \mathcal{S}^G \right)$.
	\end{Thm}
	
	\begin{proof}	
		For each $k \in \mathbb{N}$, choose a state $\sigma_k$ on $\mathfrak{A}$ such that
		$$\sigma_k \left( \frac{1}{|F_k|} \sum_{g \in F_k} \Theta_g a \right) = \left\| \frac{1}{|F_k|} \sum_{g \in F_k} \Theta_g a \right\| .$$
		Let $\omega_k = \frac{1}{|F_k|} \sum_{g \in F_k} \sigma_k \circ \Theta_g$, so
		\begin{align*}
			\omega_k(x)	& = \frac{1}{|F_k|} \sum_{g \in F_k} \sigma_k (\Theta_g x) \\
			& = \sigma_k \left( \frac{1}{|F_k|} \sum_{g \in F_k} \Theta_g x \right) , \\
			\omega_k(a)	& = \sigma_k \left( \frac{1}{|F_k|} \sum_{g \in F_k} \Theta_g a \right) \\
			& = \left\| \frac{1}{|F_k|} \sum_{g \in F_k} \Theta_g a \right\| .
		\end{align*}
		This means that in order to show that $\left( \left\| \frac{1}{|F_k|} \sum_{g \in G} \Theta_g a \right\| \right)_{k = 1}^\infty$ converges to $m \left( a \vert \mathcal{S}^G \right)$, it suffices to show that $\omega_k(a) \stackrel{k \to \infty}{\to} m \left( a \vert \mathcal{S}^G \right)$. So for the remainder of this proof, we are going to be looking instead at the sequence $(\omega_k)_{k = 1}^\infty$.
		
		Let $k_1 < k_2 < \cdots$ be some sequence such that $\left( \omega_{k_n} \right)_{n = 1}^\infty$ converges in the weak*-topology to some $\omega$. We claim that $\omega$ is $\Theta$-invariant.	This follows because if $x \in \mathfrak{A}$, then
		\begin{align*}
			\left|\omega_{k_n} \left( \Theta_{g_0} x \right) - \omega_{k_n}(x)\right|	& = \frac{1}{\left|F_{k_n}\right|} \left| \left( \sum_{g \in F_{k_n}} \sigma_{k_n} \left(\Theta_g \Theta_{g_0} x\right) \right) - \left( \sum_{g \in F_{k_n}} \sigma_{k_n} (\Theta_g x) \right) \right| \\
			& = \frac{1}{\left|F_{k_n}\right|} \left| \left( \sum_{g \in F_{k_n}} \sigma_{k_n}(\Theta_{g g_0} x) \right) - \left( \sum_{g \in F_{k_n}} \sigma_{k_n} (\Theta_g x) \right) \right| \\
			& = \frac{1}{|F_{k_n}|} \left| \left( \sum_{g \in F_{k_n} g_0} \sigma_{k_n}(\Theta_g x) \right) - \left( \sum_{g \in F_{k_n}} \sigma_{k_n} (\Theta_g x) \right) \right| \\
			& = \frac{1}{|F_{k_n}|} \left| \left( \sum_{g \in F_{k_n} g_0 \setminus F_{k_n}} \sigma_{k_n}(\Theta_g x) \right) - \left( \sum_{g \in F_{k_n} \setminus F_{k_n} g_0} \sigma_{k_n} (\Theta_g x) \right) \right| \\
			& \leq \frac{1}{|F_{k_n}|} \left( \left(\sum_{g \in F_{k_n} g_0 \setminus F_{k_n}} \|x\| \right) + \left(\sum_{g \in F_{k_n} \setminus F_{k_n} g_0} \|x\| \right) \right) \\
			& = \frac{|F_{k_n} g_0 \Delta F_{k_n}|}{|F_{k_n}|} \| x \| \\
			& \stackrel{n \to \infty}{\to} 0 .
		\end{align*}
		Therefore every limit point $\omega$ of $(\omega_k)_{k = 1}^\infty$ is $\Theta$-invariant, i.e. $\omega \in \mathcal{S}^G$.
		
		To see that $(\omega_k(a)_{k = 1}^\infty = \left( \left\| \frac{1}{|F_k|} \sum_{g \in G} \Theta_g \iota(a) \right\| \right)_{k = 1}^\infty$ converges to $m \left( a \vert \mathcal{S}^G \right)$, it will suffice to show that every limit point $\omega$ of $\left( \omega_k : k \in \mathbb{N} \right)$ satisfies 
		$$\omega \in \mathcal{S}_\mathrm{max}^G (a) .$$
		This follows because if there existed a subsequence $k_1 < k_2 < \cdots$ of $(\omega_k)_{k = 1}^\infty$ such that $\omega_{k_n}(a) \stackrel{n \to \infty}{\to} z \neq m \left( a \vert \mathcal{S}^G \right)$, then by compactness, that subsequence $\left( \omega_{k_n} : n \in \mathbb{N} \right)$ would have some subsequence converging to some $\omega_0$ for which $\omega_0(a) = z \neq m \left( a \vert \mathcal{S}^G \right)$, meaning in particular that $\omega_0 \not \in \mathcal{S}_\mathrm{max}^G(a)$.
		
		So let $k_1 < k_2 < \cdots$ be some sequence such that $\left( \omega_{k_n} \right)_{n = 1}^\infty$ converges in the weak*-topology to some $\omega$. As has already been remarked, we have that $\omega \in \mathcal{S}^G$, so $\omega(a) \leq m \left( a \vert \mathcal{S}^G \right)$. We prove the opposite inequality. Let $\phi \in \mathcal{S}^G$. Then
		\begin{align*}
			\phi(a)	& = \phi \left( \frac{1}{\left| F_{k_n} \right|} \sum_{g \in F_{k_n} } \Theta_g a \right)	& \left( \textrm{$\phi$ is $\Theta$-invariant} \right) \\
			& \leq \left\| \frac{1}{\left| F_{k_n} \right|} \sum_{g \in F_{k_n}} \Theta_g a \right\| \\
			& = \omega_{k_n}(a) \\
			\Rightarrow \phi(a)	& \leq \lim_{n \to \infty} \omega_{k_n}(a) \\
			& = \omega(a) .
		\end{align*}
		Therefore $\omega(a) \geq \sup_{\psi \in \mathcal{S}^G } \psi(a) = m \left( a \vert \mathcal{S}^G \right)$. This establishes the desired identity.
	\end{proof}

	\begin{Rmk}
	An alternate proof of Theorem \ref{Gauge exists for C*-dynamical systems} in the language of nonstandard analysis will appear in an upcoming article \cite{Spinoff}.
	\end{Rmk}

	\begin{Def}
	Given a C*-dynamical system $(\mathfrak{A}, G, \Theta)$, a positive element $a \in \mathfrak{A}$, and a right Følner sequence $\mathcal{F} = (F_k)_{k = 1}^\infty$ for $G$, we define the \emph{gauge} of $a$ to be the limit
	$$\Gamma(x) : = \lim_{k \to \infty} \left\| \frac{1}{|F_k|} \sum_{g \in F_k} \Theta_g x \right\| .$$
	\end{Def}

	Theorem \ref{Gauge exists for C*-dynamical systems} shows that the gauge exists, but Theorem \ref{Amenable gamma identity} demonstrates the way that the gauge interacts with a W*-dynamical system and a C*-model. Moreover, the gauge is dependent only on $(\mathfrak{A}, G, \Theta)$, and independent of the right Følner sequence $\mathcal{F} = (F_k)_{k = 1}^\infty$. As such, even though the gauge as we have described it is computed using a right Følner sequence $\mathcal{F} = (F_k)_{k = 1}^\infty$, we do not need to include $\mathcal{F}$ in our notation for $\Gamma$.
	
	\begin{Cor}\label{Ergodic optimization through non-surjective *-homomorphisms}
		Let $\left( \mathfrak{A} , G , \Theta \right) , \left( \tilde{\mathfrak{A}} , G , \tilde{\Theta} \right)$ be two C*-dynamical systems, and let $\pi : \mathfrak{A} \to \tilde{\mathfrak{A}}$ be a *-homomorphism (not necessarily surjective) such that
		\begin{align*}
			\tilde{\Theta}_g \circ \pi	& = \pi \circ \Theta_g	& (\forall g \in G) .
		\end{align*}
		Let $\tilde{\mathcal{S}}^G$ denote the space of $\tilde{\Theta}$-invariant states on $\tilde{\Theta}$. Then $m \left( \pi(a) \vert \tilde{\mathcal{S}}^G \right) = m \left( a \vert \operatorname{Ann}(\ker \pi) \right)$.
	\end{Cor}

	\begin{proof}
	Let $\mathfrak{B} = \pi(\mathfrak{A})$, and let $H : G \to \operatorname{Aut}(\mathfrak{B})$ be the action $H_g = \tilde{\Theta}_g \vert_{\mathfrak{B}}$. Let $K$ denote the space of all $H$-invariant states on $\mathfrak{B}$. Then
	\begin{align*}
	m \left( \pi(a) \vert \tilde{\mathcal{S}}^G \right)	& = \Gamma_{ \tilde{\mathfrak{A}} } (\pi(a))	& \mathrm{\left(Theorem \ref{Gauge exists for C*-dynamical systems}\right)}	\\
		& = \Gamma_{ \mathfrak{B} } (\pi(a)) \\ 
		& = m \left( \pi(a) \vert K \right)	& \mathrm{\left(Theorem \ref{Gauge exists for C*-dynamical systems}\right)} \\
		& = m \left( a \vert \operatorname{Ann}(\ker \pi) \right)	& \mathrm{\left(Theorem \ref{Ergodic optimization through *-homomorphisms}\right)} .
	\end{align*}
	
	\end{proof}
	
	\begin{Cor}\label{Amenable gamma identity}
	Let $(\mathfrak{M}, \rho, G, \Xi)$ be a W*-dynamical system, and let $(\mathfrak{A}, G, \Theta; \iota)$ be a C*-model of $(\mathfrak{M}, \rho, \mathbb{Z}, \Xi)$. Then if $a \in \mathfrak{A}$ is a positive element, then
	$$\Gamma(\iota(a)) = m \left( a \vert \operatorname{Ann}(\ker \iota) \right) .$$
	\end{Cor}
	
	\begin{proof}
	Write $\tilde{\mathfrak{A}} = \iota(\mathfrak{A}) \subseteq \mathfrak{M}$, and let $\tilde{\Theta} : G \to \operatorname{Aut} \left( \tilde{\mathfrak{A}} \right)$ be the action $\tilde{ \Theta }_g = \Xi_g \vert_{ \tilde{\mathfrak{A}} }$ obtained by restricting $\Xi$ to $\tilde{\mathfrak{A}}$. Write $\tilde{\mathcal{S}}^G$ for the space of $\tilde{\Theta}$-invariant states on $\tilde{\mathfrak{A}}$.
		
	We know $\Gamma_{\mathfrak{M}}(\iota(a)) = \Gamma_{\tilde{\mathfrak{A}}}(\iota(a))$. By Theorem \ref{Gauge exists for C*-dynamical systems}, we know that $\Gamma_{ \tilde{\mathfrak{A}} }(\iota(a)) = m \left( \iota(a) \vert \tilde{\mathcal{S}}^G \right)$, and by Theorem \ref{Ergodic optimization through *-homomorphisms}, we know that $m \left( \iota(a) \vert \tilde{\mathcal{S}}^G \right) = m \left( a \vert \operatorname{Ann}(\ker \iota) \right)$.
	\end{proof}
	
	This brings us to our characterization of unique ergodicity with respect to the gauge.
	
	\begin{Thm}\label{Amenable strict ergodicity and gauge}
		Let $(\mathfrak{M}, \rho, G, \Xi)$ be a W*-dynamical system, and let $(\mathfrak{A}, G, \Theta; \iota)$ be a faithful C*-model of $(\mathfrak{M}, \rho, G, \Xi)$. Then the following conditions are related by the implications (i)$\iff$(ii)$\Rightarrow$(iii).
		\begin{enumerate}[label=(\roman*)]
			\item The C*-dynamical system $(\mathfrak{A}, G, \Theta)$ is uniquely ergodic.
			\item The C*-dynamical system $(\mathfrak{A}, G, \Theta)$ is strictly ergodic.
			\item $\Gamma(a) = \rho(\iota(a))$ for all positive $a \in \mathfrak{A}$.
		\end{enumerate}
		Further, if $\mathcal{S}^G$ is a simplex, then (iii)$\Rightarrow$(i).
	\end{Thm}
	
	\begin{proof}
		(i)$\Rightarrow$(ii) Suppose that $(\mathfrak{A}, G, \Theta)$ is uniquely ergodic. Then $\rho \circ \iota$ is an invariant state on $\mathfrak{A}$, so it follows that $\rho \circ \iota$ is the unique invariant state. But $\rho \circ \iota$ is also faithful, so it follows that $(\mathfrak{A}, G, \Theta)$ is strictly ergodic.
		
		(ii)$\Rightarrow$(i) Trivial.
		
		(i)$\Rightarrow$(iii) Suppose that $(\mathfrak{A}, G, \Theta)$ is uniquely ergodic, and let $a \in \mathfrak{A}$ be positive. Let $\phi$ be an $\left( a \vert \mathcal{S}^G \right)$-maximizing state on $\mathfrak{A}$. Then $\phi = \rho \circ \iota$, since both are invariant states and $(\mathfrak{A}, G, \Theta)$ is uniquely ergodic. Then $\phi = \rho \circ \iota$, so $\Gamma(a) = \phi(a) = \rho(\iota(a))$.
		
		(iii)$\Rightarrow$(i) Suppose that $\mathcal{S}^G$ is a simplex, but that $(\mathfrak{A}, G, \Theta)$ is \emph{not} uniquely ergodic. Let $\phi \in \mathcal{S}^G$ be an extreme point of $\mathcal{S}^G$ different from $\rho \circ \iota$. Then by Corollary \ref{Uniquely maximizing states}, there exists $a \in \mathfrak{A}$ self-adjoint such that $\{ \phi \} = \mathcal{S}_\mathrm{max}^G(a)$. We can assume that $a$ is positive, since otherwise we could replace $a$ with $a + r$ for a sufficiently large positive real number $r > 0$, and $\mathcal{S}_\mathrm{max}^\mathbb{Z}(a) = \mathcal{S}_\mathrm{max}^\mathbb{Z}(a + r)$. Then $\Gamma(a) = \phi(a)$. But by the assumption that $\phi$ is uniquely $\left( a \vert \mathcal{S}^G \right)$-maximizing, it follows that $\rho(\iota(a)) < \phi(a)$. Therefore $\Gamma(a) \neq \rho(\iota(a))$, meaning that (iii) does not attain. Thus $\neg$(i)$\Rightarrow \neg$(iii).
	\end{proof}
	
	We also include the following corollary in the setting of topological dynamical systems.
	
	\begin{Cor}
		Let $(X, G, U)$ be a topological dynamical system, and let $\mu$ be a $U$-invariant, strictly positive Borel probability measure on $X$. Then the following conditions are equivalent.
		\begin{enumerate}[label=(\roman*)]
			\item The topological dynamical system $(X, G, U)$ is uniquely ergodic.
			\item The topological dynamical system $(X, G, U)$ is strictly ergodic.
			\item $\Gamma(f) = \int f \mathrm{d} \mu$ for all positive $f \in C(X)$, where $\Gamma$ is defined with respect to the W*-dynamical system $\left( L^\infty(X, \mu) , f \mapsto \int f \mathrm{d} \mu , G, U \right)$.
		\end{enumerate}
	\end{Cor}
	
	\begin{proof}
		This is a special case of Theorem \ref{Amenable strict ergodicity and gauge}. Consider the C*-dynamical system $\left( C(X), G, \Theta \right)$, where $\Theta$ is the action $\Theta_g : f \mapsto f \circ U_g$, and let $\left( L^\infty(X, \mu), f \mapsto \int f \mathrm{d} \mu, G, \Xi \right)$ be the W*-dynamical system arising from the action $\Xi_g : f \mapsto f \circ U_g$. Finally, let $\iota : C(X) \hookrightarrow L^\infty(X, \mu)$ be the embedding $f \mapsto f$. Since $\mu$ is strictly positive, the embedding $\iota$ is injective, making $(C(X), G, \Theta; \iota)$ a faithful C*-model of $\left( L^\infty(X, \mu), f \mapsto \int f \mathrm{d} \mu, G, \Xi \right)$. Finally, the space of $\Theta$-invariant states on $C(X)$ (or equivalently, the space of $U$-invariant Borel probability measures on $X$) is a simplex (a classical fact, see e.g. \cite[Lemma 2]{JenkinsonEvery}). Therefore Theorem \ref{Amenable strict ergodicity and gauge} applies in its strong form.
	\end{proof}
	
	\section*{Acknowledgments}
	
	This paper is written as part of the author's graduate studies. He is grateful to his beneficent advisor, professor Idris Assani, for no shortage of helpful guidance.
	
	\bibliography{Bibliography}
\end{document}